\documentclass[a4paper,11pt, twoside]{article}
\usepackage{amsmath,amsthm}
\usepackage{amssymb,latexsym}
\usepackage{mathrsfs}
\usepackage{enumerate}
\usepackage[colorlinks,citecolor=red]{hyperref}
\usepackage[numbers,sort&compress]{natbib}
\usepackage{indentfirst}
\usepackage{mathtools}
\usepackage{paralist,bbding,pifont}
\newtheorem{theorem}{Theorem}[section]

\newtheorem{lemma}{Lemma}[section]
\newtheorem{proposition}{Proposition}[section]
\newtheorem{definition}{Definition}[section]
\newtheorem{remark}{Remark}[section]

\theoremstyle{definition} \theoremstyle{remark}
\numberwithin{equation}{section}
\allowdisplaybreaks

\begin{document}

\markboth{Y.Z. Yang, Y. Zhou, X.L. Liu et al}{Fractional coupled chemotaxis-fluid equations}

\date{}

\baselineskip 0.22in

\title{\bf Cauchy problems for time-space fractional coupled chemotaxis-fluid equations in Besov-Morrey spaces}

\author{ Yong Zhen Yang$^1$, Yong Zhou$^{1,2}$, Xiao Lin Liu$^3$\\[1.8mm]
\footnotesize {Correspondence: yozhou@must.edu.mo}\\
\footnotesize  {$^{1}$ Faculty of Mathematics and Computational Science, Xiangtan University}\\
\footnotesize  {Hunan 411105, P.R. China}\\[1.5mm]
\footnotesize {$^2$ Macao Centre for Mathematical Sciences, Macau University of Science and Technology}\\
\footnotesize {Macau 999078, P.R. China}\\[1.5mm]
\footnotesize  {$^{3}$ College of Mathematics and Information Sciences, Guangzhou University}\\
\footnotesize  {Guangzhou 510006, P.R. China}
}

\maketitle

\begin{abstract}
In this paper, we consider the Cauchy problems for the time-space fractional coupled chemotaxis-fluid equations, which is a generalized form of the coupled chemotaxis-fluid equations studied in \cite{M.H. Yang}. In contrast to \cite{M.H. Yang}, the solution operator of the system does not satisfy the semigroup effect, which makes the approach of \cite{M.H. Yang} inapplicable. Based on the theory of harmonic analysis, using techniques such as real interpolation, embedding in Besov-Morrey spaces, the multiplier theorem, and the Hardy-Littelwood inequality in Morrey spaces, we establish global existence. As an application, we analysis the asymptotic behavior of the solutions.\\[2mm]
{\bf MSC:} 26A33; 3408A.\\
{\bf Keywords:} coupled chemotaxis-fluid equations; global existence; asymptotic behavior
\end{abstract}

\baselineskip 0.25in

\section{Introduction}
Over the past decade, the study of fractional partial differential equations has received extensive attention both domestically and internationally. Fractional calculus, characterized by its non-local properties, is particularly suited for modeling materials and processes in the real world that exhibit memory and hereditary properties. Examples include viscoelastic mechanics models, electromagnetic wave propagation, fractional control systems and controllers, percolation in fractal media, combustion phenomena, and the evolution of economic variables over time, all of which are described using fractional partial differential equations, see \cite{Y zhou2,Hilfer,A.A.Kilbas,He1,He2,K. Diethelm,Kempainen,LiLei,He3,Y.Z. Yang}.

In this article, we examine the time-space fractional Keller-Segel-Navier-Stokes model, which is described by the following system:
\begin{align}\label{Eq:K-S-N-S}
    \begin{cases}
        \partial_{t}^{\alpha}u + \left(u \cdot \nabla\right) u + \left(-\Delta\right)^{\frac{\beta}{2}}u + \nabla\pi = -v \nabla\phi, & \text{in } (0,\infty) \times \mathbb{R}^{d},\\
        \partial_{t}^{\alpha}v + \left(u \cdot \nabla\right) v + \left(-\Delta\right)^{\frac{\beta}{2}}v = -\nabla \cdot \left(v \nabla \left(-\Delta\right)^{\frac{\beta}{2}-1}w\right), & \text{in } (0,\infty) \times \mathbb{R}^{d}, \\
        \partial_{t}^{\alpha}w + \left(u \cdot \nabla\right) w + \left(-\Delta\right)^{\frac{\beta}{2}}w = \left(\left(-\Delta\right)^{\frac{2-\beta}{2}}w\right)v, & \text{in } (0,\infty) \times \mathbb{R}^{d}, \\
        \nabla \cdot u = 0, & \\
        (u,v,w)|_{t=0} = (u_{0},v_{0},w_{0}), & \text{in } \mathbb{R}^{d},
    \end{cases}
\end{align}
where \(0 < \alpha < 1\), \(1 < \beta <2\), and the functions \(u(\cdot,\cdot)\), \(v(\cdot,\cdot)\), \(w(\cdot,\cdot)\), and \(\pi(\cdot,\cdot)\), mapping from \([0, \infty) \times \mathbb{R}^{d}\) to \(\mathbb{R}^{d}\), \(\mathbb{R}\), \(\mathbb{R}\), and \(\mathbb{R}\) respectively, denote the velocity field of the fluid, the density of cells or microbes, the concentration of oxygen, and the scalar pressure. The function $\phi$ represents potential functions generated by various physical mechanisms, such as gravity or centrifugal force, and it is typically independent of $t$. The operator \(\partial_{t}^{\alpha}\) represents the \(\alpha\)-Caputo fractional derivative, and \(\left(-\Delta\right)^{\frac{\beta}{2}}\) is the fractional Laplacian operator, defined via Fourier multipliers as
\begin{equation*}
\left(-\Delta\right)^{\frac{\beta}{2}}\rho(x) = \mathcal{F}^{-1}\left(\lvert \xi \rvert^{\beta} \widehat{\rho}(\xi)\right)(x),
\end{equation*}
where the fractional Laplacian \(\left(-\Delta\right)^{\frac{\beta}{2}}\rho\) characteristically represents a L\'{e}vy diffusion process.
Moreover, the term \(\left(-\Delta\right)^{\frac{\beta}{2}-1}w\) is defined through convolution with a singular Riesz kernel,
\begin{equation*}
\left(\left(-\Delta\right)^{\frac{\beta}{2}-1}w\right)(x) = s_{d} \int_{\mathbb{R}^{d}} \frac{w(y)}{\lvert x-y\rvert^{d+\beta-2}} \, dy,
\end{equation*}
which is the attractive kernel as \cite{D. Li} pointed out.

The system \eqref{Eq:K-S-N-S} originates from the study of biomathematics in the context of the chemotaxis-fluid system, which is a chemotaxis system influenced by fluid dynamics as proposed by Tuval \cite{Tuval}. The model proposed by Tuval can be using to simulate the aggregation behavior of aerobic bacteria in an incompressible fluid. The representation is as follows:

\begin{align}
    \begin{cases}
        \partial_{t}u + \kappa\left(u \cdot \nabla\right) u = \mu\Delta u - \nabla\pi - v \nabla\phi, & \text{in } (0,\infty) \times \mathbb{R}^{d},\\
        \partial_{t}v + \left(u \cdot \nabla\right) v = \nu\Delta v - \nabla \cdot \left(\chi(w)v \nabla B(w)\right), & \text{in } (0,\infty) \times \mathbb{R}^{d}, \\
        \partial_{t}w + \left(u \cdot \nabla\right) w = \omega\Delta w - f(w)v, & \text{in } (0,\infty) \times \mathbb{R}^{d}, \\
        \nabla \cdot u = 0, & \\
        (u,v,w)|_{t=0} = (u_{0},v_{0},w_{0}), & \text{in } \mathbb{R}^{d},
    \end{cases}
    \label{chemotaxis-fluid system}
\end{align}
where the unknown functions $u, v, w, \pi$ represent the same as in \eqref{Eq:K-S-N-S}, $\chi(w)$ denotes chemotactic sensitivity, $f(w)$ is the consumption rate of cells for the chemical substance, and $B(w)$ is a linear function of $w$, such as $B(w)=w$. A significant body of literature has deeply studied the \eqref{chemotaxis-fluid system}, see \cite{M. Chae,M. Winkler,M.H. Yang,J. Zhang,J.G. Liu,H. Kozono3}. Chae \cite{M. Chae} study the stability and asymptotic behaviors on \eqref{chemotaxis-fluid system}. Duan \cite{R. Duan} established the global existence and convergence rate for classical solutions of the \eqref{chemotaxis-fluid system} near constant states. In \cite{J.G. Liu}, Liu derived the global existence of weak solutions for the two-dimensional system \eqref{chemotaxis-fluid system} with large initial data. Furthermore, for the chemotaxis-fluid system in three-dimensional space where the cell density exhibits nonlinear diffusion, they established the global existence of weak solutions.
It is worth mentioning that,  Yang \cite{M.H. Yang}  studies a simplified form of the chemotaxis-fluid system \eqref{chemotaxis-fluid system}, which can be represented as follows:
\begin{align}\label{jianhuachemotaxis-fluid system}
\begin{cases}
        \partial_{t}u + (u \cdot \nabla) u = \Delta u - \nabla\pi - v \nabla\phi, & \text{in } (0,\infty) \times \mathbb{R}^{d},\\
        \partial_{t}v + (u \cdot \nabla) v = \Delta v - \nabla \cdot (v \nabla w), & \text{in } (0,\infty) \times \mathbb{R}^{d}, \\
        \partial_{t}w + (u \cdot \nabla) w = \Delta w - w v, & \text{in } (0,\infty) \times \mathbb{R}^{d}, \\
        \nabla \cdot u = 0, & \\
        \left(u(0),v(0),w(0)\right) = (u_{0},v_{0},w_{0}), & \text{in } \mathbb{R}^{d},
    \end{cases}
\end{align}
and establishes the global well-posedness for \eqref{jianhuachemotaxis-fluid system} when the initial data
\[
\left(u_{0},v_{0},\nabla w_{0},\nabla\phi\right)\in\dot{\mathcal{N}}^{-1+\frac{d-\lambda}{r_{1}}}_{r_{1},\lambda,\infty}\times
\dot{\mathcal{N}}^{-2+\frac{d-\lambda}{r_{2}}}_{r_{2},\lambda,\infty}\times
\dot{\mathcal{N}}^{-1+\frac{d-\lambda}{r_{3}}}_{r_{3},\lambda,\infty}\times\mathcal{M}_{d-\lambda,\lambda}, c_{0}\in L^{\infty}(\mathbb{R}^{d}),
\]
improving to some extent on the global well-posedness results established by Kozono \cite{H. Kozono3} for the chemotaxis-fluid system \eqref{chemotaxis-fluid system} when the data
\[
\left(u_{0},v_{0},\nabla w_{0},\nabla\phi\right)\in L^{d,\infty}(\mathbb{R}^{d})\times L^{\frac{d}{2},\infty}(\mathbb{R}^{d})\times L^{d,\infty}(\mathbb{R}^{d})\times L^{d,\infty}(\mathbb{R}^{d}), c_{0}\in L^{\infty}(\mathbb{R}^{d}), d\geq 3,
\]
and
\[
\left(u_{0},v_{0},\nabla w_{0},\nabla\phi\right)\in L^{2,\infty}(\mathbb{R}^{d})\times L^{1}(\mathbb{R}^{d})\times L^{2,\infty}(\mathbb{R}^{d})\times L^{2,\infty}(\mathbb{R}^{d}), c_{0}\in L^{\infty}(\mathbb{R}^{d}), d=2.
\]
This process chiefly relies on the semigroup proposition of the solution operator $e^{t\Delta}$ for the system \eqref{jianhuachemotaxis-fluid system} and the equivalence in Besov-Morrey space due to the smoothing effect of the heat kernel $e^{t\Delta}$, that is for $1\leq p\leq\infty$, $0\leq \lambda<d$, $s>0$,
\begin{align}\label{HKSE}
f\in\dot{\mathcal{N}}^{-2s}_{p,\lambda,q}\Leftrightarrow
\begin{cases}
\left[\int_{0}^{\infty}\left(t^{s}\|e^{t\Delta}f\|_{p,\lambda}\right)^{q}\frac{dt}{t}\right]^{\frac{1}{q}}<\infty, & 1\leq q<\infty,\\
\sup_{t>0}\left(t^{s}\|e^{t\Delta}f\|_{p,\lambda}\right)<\infty, & q=\infty.
\end{cases}
\end{align}
Subsequently, Zhang \cite{J. Zhang} explored the following doubled chemotaxis model in the Navier-Stokes fluid involving five unknown functions \(u, v, w, n, \pi\):

\begin{align*}
\begin{cases}
\partial_{t}u + (u \cdot \nabla) u = \Delta u - \nabla\pi - v f, & \text{in } (0,\infty) \times \mathbb{R}^{d},\\
\partial_{t}v + (u \cdot \nabla) v = \Delta v - \nabla \cdot (v \nabla w) - \nabla \cdot (v \nabla n), & \text{in } (0,\infty) \times \mathbb{R}^{d}, \\
\partial_{t}w + (u \cdot \nabla) w = \Delta w - w v, & \text{in } (0,\infty) \times \mathbb{R}^{d}, \\
\partial_{t}n + (u \cdot \nabla) n = \Delta n - \gamma n + v, & \text{in } (0,\infty) \times \mathbb{R}^{d}, \\
\nabla \cdot u = 0, & \\
(u,v,w,n)|_{t=0} = (u_{0},v_{0},w_{0},n_{0}), & \text{in } \mathbb{R}^{d}.
\end{cases}
\end{align*}
By exploiting the semigroup properties of the heat kernel \(e^{t\Delta}\) and the equivalence of the norms of the heat kernel \(e^{t\Delta}\) in the Besov-Morrey space, using the approach presented in \cite{M.H. Yang}, global existence and asymptotic behavior of solutions for the doubled chemotaxis model in the Navier-Stokes fluid have been obtained.

Our model \eqref{Eq:K-S-N-S} generalizes the chemotaxis-fluid model \eqref{jianhuachemotaxis-fluid system}. It is obvious that when we consider $\alpha=1$ and $\beta=2$, \eqref{Eq:K-S-N-S} reduces to \eqref{jianhuachemotaxis-fluid system}. In contrast to \eqref{jianhuachemotaxis-fluid system}, the solution operators \(\mathcal{S}_{\alpha,\beta}(t)\) and \(\mathcal{P}_{\alpha,\beta}(t)\) determined by \eqref{Eq:K-S-N-S} do not possess the semigroup property, this leads to the fact that we cannot establish the equivalence representation similar to \eqref{HKSE}  of the norms of the solution operators \(\mathcal{S}_{\alpha,\beta}(t)\) and \(\mathcal{P}_{\alpha,\beta}(t)\) in the Besov-Morrey space. Therefore, the techniques employed in \cite{M.H. Yang, J. Zhang} are entirely ineffective for our model \eqref{Eq:K-S-N-S}. As a result, we need to employ new techniques to establish the global existence and asymptotic behavior of solutions for the model \eqref{Eq:K-S-N-S}. Moreover, In contrast to the heat kernel \(e^{t\Delta}\), the fractional heat kernel \(e^{-t(-\Delta)^{\frac{\beta}{2}}}\), corresponding to the symbol \(e^{-t|\xi|^\beta}\), does not possess smoothness at \(\xi=0\). As a result, the fractional heat kernel \(e^{-t(-\Delta)^{\frac{\beta}{2}}}\) lacks the \(C^{\infty}\) smoothness property of \(e^{t\Delta}\). Therefore, it is necessary to establish new time-space estimates based on the fractional heat kernel \(e^{-t(-\Delta)^{\frac{\beta}{2}}}\) in Besov-Morrey spaces.

The paper is organized as follow. In Section 2, we introduce some of the notations required for this paper, including fractional derivative, Besov-Morrey spaces, and multiplier theory. In Section 3, we combine the multiplier theorem to provide some time-space estimates for the solution operator of system \eqref{Eq:K-S-N-S} and prove the global existence and uniqueness of system \eqref{Eq:K-S-N-S}. In Section 4, we analysis the asymptotic behavior of mild solutions to system \eqref{Eq:K-S-N-S}.

\section{Notations and Preliminaries}
Some of the symbols, definitions, and lemmas used in this article are introduced in this section.

In this paper, we denote \( C \) as a generic constant that does not depend on any variables, and \( C \) can vary from one line to another. The notation $f \lesssim g$ indicates that there exists a positive constant $C$ such that $f \leq C g$. Let \((X, \| \cdot \|_{X})\) denote a Banach space, and let \(I \subset \mathbb{R}\) be an interval. The symbol \(L^{p}(I, X)\) represents the set of all \(X\)-valued \(L^{p}\)-integrable functions, equipped with the standard norm$\|f\|_{p}^{p} = \int_{I} \|f(t)\|^{p}_{X} \, dt$. Moreover, \(\mathcal{S}(\mathbb{R}^{d})\) denotes the class of Schwartz functions, whereas \(\mathcal{S'}(\mathbb{R}^{d})\) stands for tempered distributions. The symbols \(\mathcal{L}\) and \(\mathcal{L}^{-1}\) respectively represent the Laplace transform and the inverse Laplace transform.
For any \(g \in \mathbb{R}^{d}\), let \(\mathcal{F}\) and \(\mathcal{F}^{-1}\) respectively denote its Fourier transform and inverse Fourier transform, defined by
\[
\big(\mathcal{F}g\big)(\xi) = \widehat{g}(\xi)(2\pi)^{-\frac{d}{2}}\int_{\mathbb{R}^{d}}g(x)e^{-ix\cdot\xi}\,dx,
\]
and
\[
\big(\mathcal{F}^{-1}g\big)(x) = \check{g}(x) = (2\pi)^{-\frac{d}{2}}\int_{\mathbb{R}^{d}}\hat{g}(\xi)e^{ix\cdot\xi}\,d\xi.
\]
By employing duality, we can extend the definitions of the generalized Fourier transform and inverse Fourier transform to \(\mathcal{S'}(\mathbb{R}^{d})\). That is, for any \(g \in \mathcal{S'}\) and \(\mathcal{F}g\), \(\mathcal{F}^{-1}g\) can be defined as follows. For every \(\phi \in \mathcal{S}\),
\[
\big\langle\mathcal{F}g,\phi\big\rangle = \langle g,\mathcal{F}\phi\rangle,
\]
and
\[
\big\langle\mathcal{F}^{-1}g,\phi\big\rangle = \langle g,\mathcal{F}^{-1}\phi\rangle.
\]
Furthermore, when $f$ is a measurable function on $\mathbb{R}^{d}$ with a bound on its polynomial growth at infinity, we describe the operator $f(D)$ through the formula $f(D)a \equiv \mathcal{F}^{-1}(f\mathcal{F}a)$.

Therefore, $e^{-t(-\Delta)^{\frac{\beta}{2}}}$ is a pseudodifferential operator with the symbol $e^{-t|\xi|^{\beta}}$, and it can be defined as a convolution operator represented by the fractional heat kernel $K^{t}(x)$,
\begin{equation*}
e^{-t(-\Delta)^{\frac{\beta}{2}}}f(x) = \mathcal{F}^{-1}\left(e^{-t|\xi|^{\beta}}\mathcal{F}f\right)(x) = \mathcal{F}^{-1}(e^{-t|\xi|^{\beta}}) \ast f(x) = K_{t} \ast f(x),
\end{equation*}
where
\begin{equation*}
K_{t}(x) = \mathcal{F}^{-1}(e^{-t|\xi|^{\beta}})(x) = t^{-\frac{d}{\beta}}K\left(t^{-\frac{1}{\beta}}x\right), \; K(x) = (2\pi)^{-\frac{d}{2}}\int_{\mathbb{R}^{d}}e^{ix\xi}e^{-|\xi|^{\beta}}d\xi.
\end{equation*}
Similarly, for any $\vartheta \geq 0$, $(-\Delta)^{\frac{\vartheta}{2}}e^{-t(-\Delta)^{\frac{\beta}{2}}}$ is a pseudo-differential operator with the symbol $|\xi|^{\vartheta}e^{-t|\xi|^{\beta}}$, and can be defined as
\begin{equation*}
(-\Delta)^{\frac{\vartheta}{2}}e^{-t(-\Delta)^{\frac{\beta}{2}}}f(x) = \mathcal{F}^{-1}\left(|\xi|^{\vartheta}e^{-t|\xi|^{\beta}}\right) \ast f(x) = K_{t,\vartheta} \ast f(x),
\end{equation*}
where
\begin{equation*}
K_{t,\vartheta}(x) = t^{-\frac{\vartheta}{\beta}}t^{-\frac{d}{\beta}}K_{\vartheta}\left(t^{-\frac{1}{\beta}}x\right), \; K_{\vartheta}(x) = (2\pi)^{-\frac{d}{2}}\int_{\mathbb{R}^{d}}e^{ix\xi}|\xi|^{\vartheta}e^{-|\xi|^{\beta}}d\xi.
\end{equation*}
For further details on $e^{-t(-\Delta)^{\frac{\beta}{2}}}$ and $(-\Delta)^{\frac{\vartheta}{2}}e^{-t(-\Delta)^{\frac{\beta}{2}}}$, refer to \cite{Miao}.
\begin{definition}
For a function $f\in L^{1}\left(0,\infty;\mathcal{S}(\mathbb{R}^{n})\right)$, the Riemann-Liouville fractional integral with $0<\alpha<1$ is defined as
\[
{}_{0}I_{t}^{\alpha}\!f(t,x)=\frac{1}{\Gamma(\alpha)}\int_{0}^{t}(t-s)^{\alpha-1}f(s,x)ds=g_{\alpha}(t)*f(t,x),
\]
where $g_{\alpha}(t)=\frac{t^{\alpha-1}}{\Gamma(\alpha)}$.
\end{definition}
\begin{definition}
For $\alpha\in(0,1)$, the Caputo fractional derivative of order $\alpha$ for a function $f\in L^{1}\left(0,\infty;\mathcal{S}(\mathbb{R}^{n})\right)$ is defined as
\begin{equation*}
\partial^{\alpha}_{t}f(t,x)=\frac{d}{dt}\left(g_{1-\alpha}(t)*(f(t,x)-f(0,x))\right).
\end{equation*}
\end{definition}

Next, we delve into the Mittag-Leffler function, denoted as $E_{\alpha,\beta}(z)$, and the Mainardi-Wright function, represented by $M_{\alpha}(z)$. Their respective definitions are as follows:

\begin{equation*}
E_{\alpha,\beta}(\rho) = \sum_{n=0}^{\infty} \frac{\rho^{n}}{\Gamma(\alpha n + \beta)}, \qquad \text{for }\alpha, \beta > 0, \text{ and }\rho \in \mathbb{C},
\end{equation*}

\begin{equation*}
M_{\alpha}(\rho) = \sum_{n=0}^{\infty} \frac{(-\rho)^{n}}{n!\Gamma(1 - \alpha(n + 1))}, \qquad \alpha \in (0,1), \text{ and }\rho \in \mathbb{C}.
\end{equation*}

Regarding the Laplace transform of the Mittag-Leffler function, it is expressed as:

\begin{equation}
\int_{0}^{\infty} e^{-st} \cdot t^{\beta-1} \cdot E_{\alpha,\beta}(\pm at^{\alpha}) \, dt = \frac{s^{\alpha-\beta}}{s^{\alpha} \mp a}, \qquad \text{where }\alpha, \beta > 0.
\end{equation}
Moreover, the solution operators represented by the Mittag-Leffler functions \(\mathcal{S}_{\alpha,\beta}(t)\) and \(\mathcal{P}_{\alpha,\beta}(t)\) can be respectively reinterpreted through the integration of the Mainardi-Wright functions and the fractional heat kernel \(e^{-t(-\Delta)^{\frac{\beta}{2}}}\). These integrations are expressed as follows:
\begin{equation}\label{couple relation 1}
\mathcal{S}_{\alpha,\beta}(t) = \int_{0}^{\infty} M_{\alpha}(\theta) e^{-\theta t^{\alpha}(-\Delta)^{\frac{\beta}{2}}} \, d\theta,
\end{equation}
\begin{equation}\label{couple relation 2}
\mathcal{P}_{\alpha,\beta}(t) = \int_{0}^{\infty} \alpha \theta M_{\alpha}(\theta) e^{-\theta t^{\alpha}(-\Delta)^{\frac{\beta}{2}}} \, d\theta,
\end{equation}
where $\mathcal{S}_{\alpha,\beta}(t)$,$\mathcal{P}_{\alpha,\beta}(t)$ is defined that
$$
\mathcal{S}_{\alpha,\beta}(t)=\mathcal{F}^{-1}\big[E_{\alpha,1}(-t^{\alpha}|\xi|^{\beta})\big],\mathcal{P}_{\alpha,\beta}(t)=\mathcal{F}^{-1}\big[E_{\alpha,
\alpha}(-t^{\alpha}|\xi|^{\beta})\big]
$$
in the sense of distribution.

Furthermore,
$$
\int_{0}^{\infty}M_{\alpha}(\theta)\theta^{\rho}d\theta=\frac{\Gamma(1+\rho)}{\Gamma(1+\alpha\rho)},\text{ for }\rho>-1.
$$
Readers interested in exploring more about fractional calculus and the details of the Mittag-Leffler functions can refer to \cite{Hilfer,A.A.Kilbas,Y zhou2}.

It is important to note that the equation \eqref{Eq:K-S-N-S} exhibits well scaling properties. Specifically, if $(u, v, w, \pi, u_{0}, v_{0}, w_{0})$ satisfies equation \eqref{Eq:K-S-N-S}, then the triplet
$$
\big(u_{\lambda}, v_{\lambda}, w_{\lambda}, \pi_{\lambda}, (u_{0})_{\lambda}, (v_{0})_{\lambda}, (w_{0})_{\lambda}\big)
$$
also satisfies equation \eqref{Eq:K-S-N-S}, where the transformations are defined as follows:
\begin{align*}
u_{\lambda} &= \lambda^{(\beta-1)}u(\lambda^{\frac{\beta}{\alpha}}t, \lambda x), & (u_{0})_{\lambda} = \lambda^{(\beta-1)}u_{0}(\lambda x),\\
v_{\lambda} &= \lambda^{2(\beta-1)}v(\lambda^{\frac{\beta}{\alpha}}t, \lambda x), & (v_{0})_{\lambda} = \lambda^{2(\beta-1)}v_{0}(\lambda x),\\
w_{\lambda} &= w(\lambda^{\frac{\beta}{\alpha}}t, \lambda x), & (w_{0})_{\lambda} = w_{0}(\lambda x),\\
\pi_{\lambda} &= \lambda^{2(\beta-1)}\pi(\lambda^{\frac{\beta}{\alpha}}t, \lambda x).
\end{align*}

Furthermore, applying the Fourier and Laplace transforms to system \eqref{Eq:K-S-N-S} yields:
\begin{align}\label{Mild solution K-S-N-V}
\begin{cases}
u = \mathcal{S}_{\alpha,\beta}(t)u_{0} - \int_{0}^{t}(t-\theta)^{\alpha-1}\mathcal{P}_{\alpha,\beta}(t-\theta)\mathbb{P}(\nabla u\cdot u + v\nabla\phi) \, d\theta,\\
v = \mathcal{S}_{\alpha,\beta}(t)v_{0} - \int_{0}^{t}(t-\theta)^{\alpha-1}\mathcal{P}_{\alpha,\beta}(t-\theta)(\nabla v\cdot u + \nabla\cdot(v\cdot\nabla B(w))) \, d\theta,\\
w = \mathcal{S}_{\alpha,\beta}(t)w_{0} - \int_{0}^{t}(t-\theta)^{\alpha-1}\mathcal{P}_{\alpha,\beta}(t-\theta)(\nabla w\cdot u + ((-\Delta)^{\frac{2-\beta}{2}}w)v) \, d\theta,
\end{cases}
\end{align}
where $B(w)=(-\Delta)^{\frac{\beta - 2}{2}}w$, and $\mathbb{P}$ represents the Leray projection operator, articulated as a $d \times d$ matrix
$$
\mathbb{P} = \big\{\mathbb{P}_{j,k}\big\}_{1\leq j,k\leq d} = \big\{\delta_{jk} + R_{j}R_{k}\big\}_{1\leq j,k\leq d},
$$
with $\delta_{jk}$ being the Kronecker delta, and $R_{j} = \partial_{j}(-\Delta)^{-\frac{1}{2}}$ representing the Riesz transform, which can also be perceived as the Hilbert transform in $d$-dimensional space.
\begin{definition}
Let \(X\) be a Banach space. If the triplet \((u,v,w)\) satisfies equation \eqref{Mild solution K-S-N-V}, then \((u,v,w)\) is referred to as a mild solution of the system \eqref{Eq:K-S-N-S}.
\end{definition}
Next, we introduce some basic properities of the Besov-Morrey spaces and the Littlewood-Paley decomposition. We begin with an overview of the Morrey spaces.

Let \(1 \leq p < \infty\) and \(0 \leq \lambda < d\). The Morrey space \(\mathcal{M}_{p,\lambda}(\mathbb{R}^{d})\) can be defined as
\[
\mathcal{M}_{p,\lambda}(\mathbb{R}^{d}) = \left\{ f \in L_{\text{loc}}^{p}(\mathbb{R}^{d}) : \left\| f \right\|_{p,\lambda} < \infty \right\},
\]
where
\[
\left\| f \right\|_{p,\lambda} = \sup_{x_{0} \in \mathbb{R}^{d}} \sup_{R > 0} R^{-\frac{\lambda}{p}} \left( \int_{B(x_{0},R)} |f(y)|^{p} \, dy \right)^{\frac{1}{p}},
\]
with \(B(x_{0},R)\) being the ball in \(\mathbb{R}^{d}\) centered at \(x_{0}\) with radius \(R\). Specifically, \(\mathcal{M}_{p,0} = L^{p}\) for \(1 < p < \infty\), \(\mathcal{M}_{\infty,\lambda} = L^{\infty}\), and \(\mathcal{M}_{1,0}\) is the set of functions consisting of finite Radon measures on \(\mathbb{R}^{d}\). The pair \((\mathcal{M}_{p,\lambda}, \left\| \cdot \right\|_{p,\lambda})\) constitutes a Banach space and possesses the following scaling property:
\[
\left\| f_{k} \right\|_{p,\lambda} = k^{-\frac{d-\lambda}{p}} \left\| f \right\|_{p,\lambda}.
\]
Moreover, based on the Morrey norm, the following homogeneous Sobolev-Morrey space can be defined. For \(s \in \mathbb{R}^{d}\), \(1 \leq p < \infty\), \(\mathcal{M}_{p,\lambda}^{s} = (-\Delta)^{\frac{s}{2}}\mathcal{M}_{p,\lambda}\) equipped with the norm \(\left\| f \right\|_{\mathcal{M}_{p,\lambda}^{s}} = \left\| (-\Delta)^{\frac{s}{2}}f \right\|_{\mathcal{M}_{p,\lambda}}\). Specifically, \(\mathcal{M}^{s}_{p,0}(\mathbb{R}^{d}) = \dot{H}^{s}_{p}\) for \(1 < p < \infty\) and \((-\Delta)^{\frac{l}{2}}\mathcal{M}_{p,\lambda}^{s} = \mathcal{M}_{p,\lambda}^{s-l}\) for \(0 \leq l \leq s\).

Next, let \(\mathcal{S}_{h}(\mathbb{R}^{d}) = \left\{f \in \mathcal{S} : \partial^{\alpha}\mathcal{F}f(0) = 0 \right\}\) for all \(\alpha \in \mathbb{N}^{d} \cup \{0\}\), and let \(\mathcal{S}^{'}_{h}(\mathbb{R}^{d}) = \mathcal{S'}/\mathcal{P}(\mathbb{R}^{d})\) is the dual space of $\mathcal{S}_{h}(\mathbb{R}^{d})$, where \(\mathcal{P}(\mathbb{R}^{d})\) denotes the space of polynomials. Let \(\left\{\phi_{j}\right\}_{j \in \mathbb{Z}}\) be a sequence of radial smooth bump functions, and
\[
\text{supp}\, \mathcal{F}\phi_{0} = \mathcal{C}_{0} := \left\{\xi \in \mathbb{R}^{d} : \frac{2}{3} \leq |\xi| \leq 3\right\} \text{ and } \phi_{0} > 0 \text{ on } \mathcal{C}_{0}.
\]
Define \(\widehat{\phi_{k}} = \widehat{\phi_{0}}(2^{-k}\xi)\), \(\mathcal{C}_{j} = 2^{j}\mathcal{C}_{0}\), and consider the following radial smooth bump functions \(\left\{\varphi_{j}\right\}_{j \in \mathbb{Z}}\) satisfying the properties:
\[
\widehat{\varphi_{k}}(\xi) =
\begin{cases}
\frac{\widehat{\phi_{k}}(\xi)}{\sum_{k \in \mathbb{Z}}\widehat{\phi_{k}}(\xi)}, & \xi \in \mathbb{R}^{d} \setminus \{0\}, \\
0, & \xi = 0,
\end{cases}
\]
and
\[
\sum_{k \in \mathbb{Z}}\widehat{\varphi_{k}}(\xi) =
\begin{cases}
1, & \xi \neq 0, \\
0, & \xi = 0.
\end{cases}
\]
Moreover,
\[
\text{supp}\, \widehat{\varphi_{k}} = 2^{k}\mathcal{C}_{k}, \quad \widehat{\varphi_{k}}(\xi) = \widehat{\varphi_{0}}(2^{-k}\xi).
\]
Based on the above Littlewood-Paley decomposition, the homogeneous Besov-Morrey space \(\dot{\mathcal{N}}^{s}_{p,\lambda,r}\) can be defined as
\[
\dot{\mathcal{N}}^{s}_{p,\lambda,r}(\mathbb{R}^{d}) = \left\{f \in \mathcal{S}^{'}_{h}(\mathbb{R}^{d}) : \left\|f\right\|_{\dot{\mathcal{N}}^{s}_{p,\lambda,r}} < \infty \right\},
\]
where
\[
\left\|f\right\|_{\dot{\mathcal{N}}^{s}_{p,\lambda,r}} =
\begin{cases}
\left(\sum_{k \in \mathbb{Z}} 2^{ksr} \left\|\Delta_{k}f\right\|_{p,\lambda}^{r}\right)^{\frac{1}{r}}, & 1\leq r < \infty,\\
\sup_{k \in \mathbb{Z}} \left(2^{ks} \left\|\Delta_{k}f\right\|_{p,\lambda}\right), & r=\infty,
\end{cases}
\]
where \(\Delta_{k}f = \varphi_{k} \ast f\). In particular, $\dot{\mathcal{N}}^{s}_{p,0,r}=\dot{B}^{s}_{p,r}$. Similarly, $\dot{\mathcal{N}}^{s}_{p,\lambda,r}(\mathbb{R}^{d})$ and $\mathcal{M}^{s}_{p,\lambda}(\mathbb{R}^{d})$ exhibit the following scaling properties, that is,
\[
\big\|g(k\cdot)\big\|_{\dot{\mathcal{N}}^{s}_{p,\lambda,r}(\mathbb{R}^{d})} = k^{s-\frac{d-\lambda}{p}}\big\|g\big\|_{\dot{\mathcal{N}}^{s}_{p,\lambda,r}(\mathbb{R}^{d})}, \quad \big\|g(k\cdot)\big\|_{\mathcal{M}^{s}_{p,\lambda}(\mathbb{R}^{d})} = k^{s-\frac{d-\lambda}{p}}\big\|g\big\|_{\mathcal{M}^{s}_{p,\lambda}(\mathbb{R}^{d})}.
\]
For further details on Morrey spaces, Sobolev-Morrey spaces, and Besov-Morrey spaces, readers are encouraged to consult the following references \cite{H. Kozono1, H. Kozono2, A. Mazzucato, T. Kato, M. Taylor}.

The following propositions, taken from \cite{A. Mazzucato, A. Kassymov, H. Kozono2}, describe essential properties:

\begin{proposition}\label{some propesitions} \hfill
\begin{itemize}
    \item[\rm(i)] For $\omega_{1}\neq \omega_{2} \in \mathbb{R}$, $\theta \in [0,1]$ and $\omega = \theta \omega_{1} + (1-\theta)\omega_{2}$ $r_{1}, r_{2}, r \in [1, \infty]$, the following interpolation identities hold true:
    \begin{align*}
    \big(\dot{\mathcal{N}}^{\omega_{1}}_{p,\lambda,r_{1}}(\mathbb{R}^{d}), \dot{\mathcal{N}}^{\omega_{2}}_{p,\lambda,r_{2}}(\mathbb{R}^{d})\big)_{\theta,r} &= \dot{\mathcal{N}}^{\omega}_{p,\lambda,r}(\mathbb{R}^{d}), \\
    \big(\mathcal{M}^{\omega_{1}}_{p,\lambda}(\mathbb{R}^{d}), \mathcal{M}^{\omega_{2}}_{p,\lambda}(\mathbb{R}^{d})\big)_{\theta,r} &= \dot{\mathcal{N}}^{\omega}_{p,\lambda,r}(\mathbb{R}^{d}).
    \end{align*}

    \item[\rm(ii)] Given $0 < \kappa < d$, with $0 < \lambda < d-\kappa p$, and setting $\frac{1}{q} = \frac{1}{p} - \frac{\kappa}{d-\lambda}$, it follows that
    \[
    \big\|I_{d-\kappa}u\big\|_{\mathcal{M}_{q,\lambda}(\mathbb{R}^{d})} \lesssim \big\|u\big\|_{\mathcal{M}_{p,\lambda}(\mathbb{R}^{d})},
    \]
    where the operator $I_{d-\kappa}$ is defined as
    \[
    I_{d-\kappa}u(x) = s_{d} \int_{\mathbb{R}^{d}} \frac{u(y)}{|x-y|^{d-\kappa}} \, dy.
    \]

    \item[\rm(iii)] For the multiplication of functions within these spaces, we have the inequality
    \[
    \big\|fg\big\|_{\mathcal{M}_{p_{3},\lambda}(\mathbb{R}^{d})} \lesssim \big\|f\big\|_{\mathcal{M}_{p_{1},\lambda}(\mathbb{R}^{d})} \big\|g\big\|_{\mathcal{M}_{p_{2},\lambda}(\mathbb{R}^{d})},
    \]
    provided that
    \[
    \frac{1}{p_{3}} = \frac{1}{p_{1}} + \frac{1}{p_{2}}.
    \]

    \item[\rm(iv)] Embedding relations among spaces are as follows:
    \begin{align*}
    \dot{\mathcal{N}}^{s_{1}}_{p_{1},\lambda,r}(\mathbb{R}^{d}) &\hookrightarrow \dot{\mathcal{N}}^{s_{2}}_{p_{2},\lambda,r}(\mathbb{R}^{d}), \\
    \mathcal{M}^{s_{1}}_{p_{1},\lambda}(\mathbb{R}^{d}) &\hookrightarrow \mathcal{M}^{s_{2}}_{p_{2},\lambda}(\mathbb{R}^{d}), \\
    \dot{\mathcal{N}}^{0}_{p,\lambda,1}(\mathbb{R}^{d}) &\hookrightarrow \mathcal{M}_{p,\lambda}(\mathbb{R}^{d}) \hookrightarrow \dot{\mathcal{N}}^{0}_{p,\lambda,\infty}(\mathbb{R}^{d}),
    \end{align*}
    with the conditions:
    \[
    s_{1} > s_{2},\quad \text{and} \quad s_{1} - \frac{d-\lambda}{p_{1}} = s_{2} - \frac{d-\lambda}{p_{2}}.
    \]
\end{itemize}
\end{proposition}
The multiplier Lemma below plays a significant role in establishing estimates for operators in Besov-Morrey spaces, as mentioned in \cite{H. Kozono2}.
\begin{lemma}\label{multiplier theorem}
Consider $1 < p < \infty$ and $s, l \in \mathbb{R}$. Let $\sigma(\xi)$ be a $C^{\left[\frac{d}{2}\right]+1}$-function on $\mathbb{R}^{d} \setminus \{0\}$ and satisfy
\[
\biggl|\frac{\partial^{|\gamma|}}{\partial\xi^{\gamma}}\sigma(\xi)\biggr| \lesssim \big|\xi\big|^{l-|\gamma|}, \quad \xi \neq 0,
\]
for any multi-index $\gamma \in \mathbb{N}^{d} \cup \{0\}$ such that $|\gamma| \leq \left[\frac{d}{2}\right] + 1$. Then, the operator $\sigma(D)$ is bounded from $\mathcal{M}^{s}_{p,\lambda}(\mathbb{R}^{d})$ to $\mathcal{M}^{s-l}_{p,\lambda}(\mathbb{R}^{d})$.
\end{lemma}
\section{Global existence on Besov-Morrey space}
This section, we construct the global existence of system \eqref{Eq:K-S-N-S} in Besov-Morrey spaces. Initially, we establish the estimates for the solution operators through lemmas.

Unlike the symbol $e^{-t|\xi|^{2}}$ for the heat kernel $G(t,x) = \mathcal{F}^{-1}(e^{-t|\xi|^{2}})$, the symbol $e^{-t|\xi|^{\beta}}$ for the fractional heat kernel $\mathcal{F}^{-1}(e^{-t|\xi|^{\beta}})$ is not smooth at $\xi=0$. Notably, Lemma \ref{multiplier theorem} does not require smoothness at $\xi=0$, which is advantageous for estimating the fractional heat kernel $e^{-t(-\Delta)^{\frac{\beta}{2}}}$. In this respect, we improved \cite{H. Kozono1,M.H. Yang,J. Zhang} where the smoothness at $\xi=0$ is needed.
\begin{lemma}\label{fractional heat kenel lemma}
Let $\vartheta \geq 0$, $0\leq\lambda<d$, $1 < p_{1} \leq p_{2} < \infty$, $s_{1} \leq s_{2}$, with $s_{1}, s_{2} \in \mathbb{R}$, and $1 \leq r \leq \infty$. We have the following estimates:
\begin{enumerate}
    \item[\rm(i)]
    \begin{equation}\label{fraction heat estimate1}
    \left\|(-\Delta)^{\frac{\vartheta}{2}}e^{-t(-\Delta)^{\frac{\beta}{2}}}f\right\|_{\mathcal{M}^{s_{2}}_{p_{2},\lambda}} \lesssim t^{-\frac{\vartheta+s_{2}-s_{1}}{\beta} - \frac{1}{\beta}\left(\frac{d-\lambda}{p_{1}} - \frac{d-\lambda}{p_{2}}\right)}\left\|f\right\|_{\mathcal{M}^{s_{1}}_{p_{1},\lambda}},
    \end{equation}
    \item[\rm(ii)]
    \begin{equation}\label{fraction heat estimate2}
    \left\|(-\Delta)^{\frac{\vartheta}{2}}e^{-t(-\Delta)^{\frac{\beta}{2}}}f\right\|_{\dot{\mathcal{N}}^{s_{2}}_{p_{2},\lambda,r}} \lesssim t^{-\frac{\vartheta+s_{2}-s_{1}}{\beta} - \frac{1}{\beta}\left(\frac{d-\lambda}{p_{1}} - \frac{d-\lambda}{p_{2}}\right)}\left\|f\right\|_{\dot{\mathcal{N}}^{s_{1}}_{p_{1},\lambda,r}},
    \end{equation}
    \item[\rm(iii)] if $s_{2} > s_{1}$,
    \begin{equation}\label{fraction heat estimate3}
    \left\|(-\Delta)^{\frac{\vartheta}{2}}e^{-t(-\Delta)^{\frac{\beta}{2}}}f\right\|_{\dot{\mathcal{N}}^{s_{2}}_{p_{2},\lambda,1}} \lesssim t^{-\frac{\vartheta+s_{2}-s_{1}}{\beta} - \frac{1}{\beta}\left(\frac{d-\lambda}{p_{1}} - \frac{d-\lambda}{p_{2}}\right)}\left\|f\right\|_{\dot{\mathcal{N}}^{s_{1}}_{p_{1},\lambda,\infty}}.
    \end{equation}
\end{enumerate}
\end{lemma}
\begin{proof}
Using induction, for every fixed $\gamma \in \mathbb{N}^{d} \cup \{0\}$ and $\xi \neq 0$, we obtain
\[
\frac{\partial^{|\gamma|}}{\partial \xi^{\gamma}}e^{-|\xi|^{\beta}} = e^{-|\xi|^{\beta}}|\xi|^{-2|\gamma|} \sum_{j=1}^{|\gamma|}P_{\gamma,j}(\xi)|\xi|^{j\beta},
\]
where $P_{\gamma,j}$ is a homogeneous polynomial of degree $|\gamma|$ for $j=1,2,\ldots,|\gamma|$. Therefore, by the Leibniz formula,
\begin{align*}
\left|\frac{\partial^{|\gamma|}}{\partial \xi^{\gamma}}\left(|\xi|^{\vartheta}e^{-|\xi|^{\beta}}\right)\right| &= \left|\sum_{a+b=\gamma}C_{\gamma}^{a} \left(\partial^{a}|\xi|^{\vartheta}\right)\sum_{j=1}^{|b|}e^{-|\xi|^{\beta}}|\xi|^{-2|b|}P_{b,j}(\xi)|\xi|^{j\beta}\right|\\
&\lesssim \sum_{a+b=\gamma}C_{\gamma}^{a} |\xi|^{-|a|-|b|}\sum_{j=1}^{|b|}e^{-|\xi|^{\beta}}|\xi|^{-|b|}P_{b,j}(\xi)|\xi|^{\vartheta + j\beta}\\
&\lesssim |\xi|^{-|\gamma|}.
\end{align*}
Note that with the change of variables $y' = y / t^{\frac{1}{\beta}}$, we have $K_{t} \ast f_{t^{-\frac{1}{\beta}}}(\cdot) = \left(K \ast f\right)_{t^{-\frac{1}{\beta}}}(\cdot)$ and $K_{t, \vartheta} \ast f_{t^{-\frac{1}{\beta}}}(\cdot) = t^{-\frac{\vartheta}{\beta}}\left(K_{\vartheta} \ast f\right)_{t^{-\frac{1}{\beta}}}(\cdot)$. Thus, by Lemma \ref{multiplier theorem}, we get
\begin{align*}
\left\|(-\Delta)^{\frac{\vartheta}{2}}e^{-t(-\Delta)^{\frac{\beta}{2}}}f\right\|_{\mathcal{M}^{s_{2}}_{p_{2},\lambda}} &= \left\|K_{\vartheta, t} \ast f\right\|_{\mathcal{M}^{s_{2}}_{p_{2},\lambda}}\\
&= \left\|t^{-\frac{\vartheta}{\beta}}\left(K_{\vartheta} \ast f_{t^{\frac{1}{\beta}}}\right)_{t^{-\frac{1}{\beta}}}\right\|_{\mathcal{M}^{s_{2}}_{p_{2},\lambda}}\\
&\lesssim t^{-\frac{\vartheta}{\beta}}t^{-\frac{1}{\beta}\left(s_{2} - \frac{d-\lambda}{p_{2}}\right)}\left\|f_{t^{\frac{1}{\beta}}}(\cdot)\right\|_{\mathcal{M}^{s_{2}}_{p_{2},\lambda}}\\
&\lesssim t^{-\frac{\vartheta}{\beta}}\left\|f\right\|_{\mathcal{M}^{s_{2}}_{p_{2},\lambda}}.
\end{align*}
Therefore, combining Proposition \ref{some propesitions}, we obtain
\begin{align*}
\left\|(-\Delta)^{\frac{\vartheta}{2}}e^{-t(-\Delta)^{\frac{\beta}{2}}}f\right\|_{\mathcal{M}^{s_{2}}_{p_{2},\lambda}} &= \left\|(-\Delta)^{\frac{\vartheta + s_{2}}{2}}e^{-t(-\Delta)^{\frac{\beta}{2}}}f\right\|_{\mathcal{M}^{0}_{p_{2},\lambda}}\\
&\lesssim\left\|(-\Delta)^{\frac{\vartheta + s_{2}}{2}}e^{-t(-\Delta)^{\frac{\beta}{2}}}f\right\|_{\mathcal{M}^{l}_{p_{1},\lambda}}\\
&\lesssim \left\|(-\Delta)^{\frac{\vartheta + s_{2} + l - s_{1}}{2}}e^{-t(-\Delta)^{\frac{\beta}{2}}}f\right\|_{\mathcal{M}^{s_{1}}_{p_{1},\lambda}}\\
&\lesssim t^{-\frac{\vartheta + s_{2} - s_{1}}{\beta} - \frac{1}{\beta}\left(\frac{d-\lambda}{p_{1}} - \frac{d-\lambda}{p_{2}}\right)}\left\|f\right\|_{\mathcal{M}^{s_{1}}_{p_{1},\lambda}},
\end{align*}
where $l = \frac{d-\lambda}{p_{1}} - \frac{d-\lambda}{p_{2}}$. Using real interpolation,
\[
\dot{\mathcal{N}}^{s_{2}}_{p_{2},\lambda,r} = \left(\mathcal{M}_{p_{2},\lambda}^{s_{21}}, \mathcal{M}_{p_{2},\lambda}^{s_{22}}\right)_{\theta, r}, \quad \dot{\mathcal{N}}^{s_{1}}_{p_{1},\lambda,r} = \left(\mathcal{M}_{p_{1},\lambda}^{s_{11}}, \mathcal{M}_{p_{1},\lambda}^{s_{12}}\right)_{\theta, r},
\]
for $s_{2} = \theta s_{21} + (1 - \theta)s_{22}$, $s_{21} \neq s_{22}$, and $s_{1} = \theta s_{11} + (1 - \theta)s_{12}$, $s_{11} \neq s_{12}$.

Thus, we have
\[
\left\|(-\Delta)^{\frac{\vartheta}{2}}e^{-t(-\Delta)^{\frac{\beta}{2}}}f\right\|_{\dot{\mathcal{N}}^{s_{2}}_{p_{2},\lambda,r}} \lesssim A_{1}^{\theta}A_{2}^{1 - \theta}\left\|f\right\|_{\dot{\mathcal{N}}^{s_{1}}_{p_{1},\lambda,r}} \lesssim t^{-\frac{\vartheta + s_{2} - s_{1}}{\beta} - \frac{1}{\beta}\left(\frac{d-\lambda}{p_{1}} - \frac{d-\lambda}{p_{2}}\right)},
\]
where
\begin{align*}
A_{1} := &\left\|(-\Delta)^{\frac{\vartheta}{2}}e^{-t(-\Delta)^{\frac{\beta}{2}}}\right\|_{\mathcal{M}^{s_{11}}_{p_{1},\lambda} \rightarrow \mathcal{M}^{s_{21}}_{p_{2},\lambda}} \lesssim t^{-\frac{\vartheta + s_{21} - s_{11}}{\beta} - \frac{1}{\beta}\left(\frac{d-\lambda}{p_{1}} - \frac{d-\lambda}{p_{2}}\right)},\\
A_{2} := &\left\|(-\Delta)^{\frac{\vartheta}{2}}e^{-t(-\Delta)^{\frac{\beta}{2}}}\right\|_{\mathcal{M}^{s_{12}}_{p_{1},\lambda} \rightarrow \mathcal{M}^{s_{22}}_{p_{2},\lambda}} \lesssim t^{-\frac{\vartheta + s_{22} - s_{12}}{\beta} - \frac{1}{\beta}\left(\frac{d-\lambda}{p_{1}} - \frac{d-\lambda}{p_{2}}\right)}.
\end{align*}
If $s_{2} > s_{1}$, using the above estimates, we have
\begin{align*}
\left\|(-\Delta)^{\frac{\vartheta}{2}}e^{-t(-\Delta)^{\frac{\beta}{2}}}f\right\|_{\dot{\mathcal{N}}^{\frac{3s_{2} - s_{1}}{2}}_{p_{2},\lambda,\infty}} &\lesssim t^{-\frac{\vartheta + \frac{3}{2}(s_{2} - s_{1})}{\beta} - \frac{1}{\beta}\left(\frac{d-\lambda}{p_{1}} - \frac{d-\lambda}{p_{2}}\right)}\left\|f\right\|_{\dot{\mathcal{N}}^{s_{1}}_{p_{1},\lambda,\infty}},\\
\left\|(-\Delta)^{\frac{\vartheta}{2}}e^{-t(-\Delta)^{\frac{\beta}{2}}}f\right\|_{\dot{\mathcal{N}}^{\frac{s_{2} + s_{1}}{2}}_{p_{2},\lambda,\infty}} &\lesssim t^{-\frac{\vartheta + \frac{1}{2}(s_{2} - s_{1})}{\beta} - \frac{1}{\beta}\left(\frac{d-\lambda}{p_{1}} - \frac{d-\lambda}{p_{2}}\right)}\left\|f\right\|_{\dot{\mathcal{N}}^{s_{1}}_{p_{1},\lambda,\infty}},
\end{align*}
noting that $\left(\dot{\mathcal{N}}^{\frac{3s_{2} - s_{1}}{2}}_{p_{2},\lambda,\infty}, \dot{\mathcal{N}}^{\frac{s_{2} + s_{1}}{2}}_{p_{2},\lambda,\infty}\right)_{\frac{1}{2},1} = \dot{\mathcal{N}}^{s_{2}}_{p_{2},\lambda,1}$. Therefore, we obtain
\[
\left\|(-\Delta)^{\frac{\vartheta}{2}}e^{-t(-\Delta)^{\frac{\beta}{2}}}f\right\|_{\dot{\mathcal{N}}^{s_{2}}_{p_{2},\lambda,1}} \lesssim t^{-\frac{\vartheta + s_{2} - s_{1}}{\beta} - \frac{1}{\beta}\left(\frac{d-\lambda}{p_{1}} - \frac{d-\lambda}{p_{2}}\right)}\left\|f\right\|_{\dot{\mathcal{N}}^{s_{1}}_{p_{1},\lambda,\infty}}.
\]
The proof is completed.
\end{proof}
\begin{lemma}\label{Mittag estimats lemma}
Given the conditions of Lemma \ref{fractional heat kenel lemma}, we obtain the following estimates.
\begin{enumerate}
\item[\rm(i)] If $\vartheta+s_{2}-s_{1}+\frac{d-\lambda}{p_{1}}-\frac{d-\lambda}{p_{2}}<\beta$,
\begin{align}\label{estimate1}
\left\|(-\Delta)^{\frac{\vartheta}{2}}\mathcal{S}_{\alpha,\beta}(t)f
\right\|_{\mathcal{M}^{s_{2}}_{p_{2},\lambda}} \lesssim t^{-\frac{\alpha}{\beta}(\vartheta+s_{2}-s_{1}) - \frac{\alpha}{\beta}\left(\frac{d-\lambda}{p_{1}} - \frac{d-\lambda}{p_{2}}\right)}\left\|f\right\|_{\mathcal{M}^{s_{1}}_{p_{1},\lambda}}.
\end{align}

\item[\rm(ii)] If $\vartheta+s_{2}-s_{1}+\frac{d-\lambda}{p_{1}}-\frac{d-\lambda}{p_{2}}<\beta$,
\begin{align}\label{estimate2}
\left\|(-\Delta)^{\frac{\vartheta}{2}}\mathcal{S}_{\alpha,\beta}(t)f
\right\|_{\dot{\mathcal{N}}^{s_{2}}_{p_{2},\lambda,r}} \lesssim t^{-\frac{\alpha}{\beta}(\vartheta+s_{2}-s_{1}) - \frac{\alpha}{\beta}\left(\frac{d-\lambda}{p_{1}} - \frac{d-\lambda}{p_{2}}\right)}\left\|f\right\|_{\dot{\mathcal{N}}^{s_{1}}_{p_{1},\lambda,r}}.
\end{align}

\item[\rm(iii)] If $\vartheta+s_{2}-s_{1}+\frac{d-\lambda}{p_{1}}-\frac{d-\lambda}{p_{2}}<\beta$ and $s_{2}>s_{1}$,
\begin{align}\label{estimate3}
\left\|(-\Delta)^{\frac{\vartheta}{2}}\mathcal{S}_{\alpha,\beta}(t)f
\right\|_{\dot{\mathcal{N}}^{s_{2}}_{p_{2},\lambda,1}} \lesssim t^{-\frac{\alpha}{\beta}(\vartheta+s_{2}-s_{1}) - \frac{\alpha}{\beta}\left(\frac{d-\lambda}{p_{1}} - \frac{d-\lambda}{p_{2}}\right)}\left\|f\right\|_{\dot{\mathcal{N}}^{s_{1}}_{p_{1},\lambda,\infty}}.
\end{align}

\item[\rm(iv)] If $\vartheta+s_{2}-s_{1}+\frac{d-\lambda}{p_{1}}-\frac{d-\lambda}{p_{2}}<2\beta$,
\begin{align}\label{estimate4}
\left\|(-\Delta)^{\frac{\vartheta}{2}}\mathcal{P}_{\alpha,\beta}(t)f
\right\|_{\mathcal{M}^{s_{2}}_{p_{2},\lambda}} \lesssim t^{-\frac{\alpha}{\beta}(\vartheta+s_{2}-s_{1}) - \frac{\alpha}{\beta}\left(\frac{d-\lambda}{p_{1}} - \frac{d-\lambda}{p_{2}}\right)}\left\|f\right\|_{\mathcal{M}^{s_{1}}_{p_{1},\lambda}}.
\end{align}

\item[\rm(v)] If $\vartheta+s_{2}-s_{1}+\frac{d-\lambda}{p_{1}}-\frac{d-\lambda}{p_{2}}<2\beta$,
\begin{align}\label{estimate5}
\left\|(-\Delta)^{\frac{\vartheta}{2}}\mathcal{P}_{\alpha,\beta}(t)f
\right\|_{\dot{\mathcal{N}}^{s_{2}}_{p_{2},\lambda,r}} \lesssim t^{-\frac{\alpha}{\beta}(\vartheta+s_{2}-s_{1}) - \frac{\alpha}{\beta}\left(\frac{d-\lambda}{p_{1}} - \frac{d-\lambda}{p_{2}}\right)}\left\|f\right\|_{\dot{\mathcal{N}}^{s_{1}}_{p_{1},\lambda,r}}.
\end{align}

\item[\rm(vi)] If $\vartheta+s_{2}-s_{1}+\frac{d-\lambda}{p_{1}}-\frac{d-\lambda}{p_{2}}<2\beta$ and $s_{2}>s_{1}$,
\begin{align}\label{estimate6}
\left\|(-\Delta)^{\frac{\vartheta}{2}}\mathcal{P}_{\alpha,\beta}(t)
\right\|_{\dot{\mathcal{N}}^{s_{2}}_{p_{2},\lambda,1}} \lesssim t^{-\frac{\alpha}{\beta}(\vartheta+s_{2}-s_{1}) - \frac{\alpha}{\beta}\left(\frac{d-\lambda}{p_{1}} - \frac{d-\lambda}{p_{2}}\right)}\left\|f\right\|_{\dot{\mathcal{N}}^{s_{1}}_{p_{1},\lambda,\infty}}.
\end{align}
\end{enumerate}
\end{lemma}

\begin{proof}
We only prove \eqref{estimate1}-\eqref{estimate3}; the proofs for \eqref{estimate4}-\eqref{estimate6} follow by a similar reasoning. By using \eqref{fraction heat estimate1} and \eqref{couple relation 1}, we have
\begin{align*}
&\left\|(-\Delta)^{\frac{\vartheta}{2}}\mathcal{S}_{\alpha,\beta}(t)f
\right\|_{\mathcal{M}^{s_{2}}_{p_{2},\lambda}}\\
&\leq\int_{0}^{\infty}M_{\alpha}(\theta)
\bigg\|(-\Delta)^{\frac{\vartheta}{2}}e^{-\theta t^{\alpha}(-\Delta)^{\frac{\beta}{2}}}f\bigg\|_{\mathcal{M}^{s_{2}}_{p_{2},\lambda}}d\theta\\
&\leq\int_{0}^{\infty}M_{\alpha}(\theta)\theta^{-\frac{1}{\beta}
(\vartheta+s_{2}-s_{1}+\frac{d-\lambda}{p_{1}}-\frac{d-\lambda}{p_{2}})}d\theta t^{-\frac{\alpha}{\beta}(\vartheta+s_{2}-s_{1}) - \frac{\alpha}{\beta}\left(\frac{d-\lambda}{p_{1}} - \frac{d-\lambda}{p_{2}}\right)}\left\|f\right\|_{\mathcal{M}^{s_{1}}_{p_{1},\lambda}}\\
&=\frac{\Gamma(1-\frac{1}{\beta}
(\vartheta+s_{2}-s_{1}+\frac{d-\lambda}{p_{1}}-\frac{d-\lambda}{p_{2}}))}{\Gamma(1-\frac{\alpha}{\beta}
(\vartheta+s_{2}-s_{1}+\frac{d-\lambda}{p_{1}}-\frac{d-\lambda}{p_{2}}))}t^{-\frac{\alpha}{\beta}(\vartheta+s_{2}-s_{1}) - \frac{\alpha}{\beta}\left(\frac{d-\lambda}{p_{1}} - \frac{d-\lambda}{p_{2}}\right)}\left\|f\right\|_{\mathcal{M}^{s_{1}}_{p_{1},\lambda}}.
\end{align*}
Using \eqref{couple relation 1}, \eqref{fraction heat estimate2}, and \eqref{fraction heat estimate3}, we obtain
\begin{align*}
&\left\|(-\Delta)^{\frac{\vartheta}{2}}\mathcal{S}_{\alpha,\beta}(t)f
\right\|_{\dot{\mathcal{N}}^{s_{2}}_{p_{2},\lambda,r}}\\
&\leq\int_{0}^{\infty}M_{\alpha}(\theta)
\bigg\|(-\Delta)^{\frac{\vartheta}{2}}e^{-\theta t^{\alpha}(-\Delta)^{\frac{\beta}{2}}}f\bigg\|_{\dot{\mathcal{N}}^{s_{2}}_{p_{2},\lambda,r}}d\theta\\
&\leq\int_{0}^{\infty}M_{\alpha}(\theta)\theta^{-\frac{1}{\beta}
(\vartheta+s_{2}-s_{1}+\frac{d-\lambda}{p_{1}}-\frac{d-\lambda}{p_{2}})}d\theta t^{-\frac{\alpha}{\beta}(\vartheta+s_{2}-s_{1}) - \frac{\alpha}{\beta}\left(\frac{d-\lambda}{p_{1}} - \frac{d-\lambda}{p_{2}}\right)}\left\|f\right\|_{\dot{\mathcal{N}}^{s_{1}}_{p_{1},\lambda,r}}\\
&=\frac{\Gamma(1-\frac{1}{\beta}
(\vartheta+s_{2}-s_{1}+\frac{d-\lambda}{p_{1}}-\frac{d-\lambda}{p_{2}}))}{\Gamma(1-\frac{\alpha}{\beta}
(\vartheta+s_{2}-s_{1}+\frac{d-\lambda}{p_{1}}-\frac{d-\lambda}{p_{2}}))}t^{-\frac{\alpha}{\beta}(\vartheta+s_{2}-s_{1}) - \frac{\alpha}{\beta}\left(\frac{d-\lambda}{p_{1}} - \frac{d-\lambda}{p_{2}}\right)}\left\|f\right\|_{\dot{\mathcal{N}}^{s_{1}}_{p_{1},\lambda,r}}.
\end{align*}
Similarly, we can also have,
$$
\left\|(-\Delta)^{\frac{\vartheta}{2}}\mathcal{S}_{\alpha,\beta}(t)f
\right\|_{\dot{\mathcal{N}}^{s_{2}}_{p_{2},\lambda,1}}\leq t^{-\frac{\alpha}{\beta}(\vartheta+s_{2}-s_{1}) - \frac{\alpha}{\beta}\left(\frac{d-\lambda}{p_{1}} - \frac{d-\lambda}{p_{2}}\right)}\left\|f\right\|_{\dot{\mathcal{N}}^{s_{1}}_{p_{1},\lambda,\infty}}.
$$
The proof is thus completed.
\end{proof}
In light of the estimations mentioned above, we proceed to demonstrate the well-posedness of system \eqref{Eq:K-S-N-S} in Besov-Morrey spaces. A common phenomenon in many partial differential equations (PDEs) is the emergence of different scales of time and space, while certain phenomena seem to remain invariant under the choice of these scales. Mathematically, the scale invariance properties of nonlinear differential equations are crucial tools for analyzing solutions, especially when the nonlinearity is a homogeneous function. Our primary approach is founded on Kato's methods; that is, we consider the contraction mapping principle in certain Banach spaces that exhibit scale invariance, as discussed in \cite{T. Kato, B.H. H}.

Let
\begin{align}\label{canshufanwei1}
\begin{cases}
r_{1}>\frac{d-\lambda}{\beta-1}, r_{2}>\frac{d-\lambda}{2\beta-2}, r_{3}>d-\lambda, q_{1}>\frac{d-\lambda}{\beta-1}, \frac{d-\lambda}{2\beta-2}<q_{2}<\frac{d-\lambda}{\beta-1}, d-\lambda<q_{3}<\frac{d-\lambda}{2-\beta},\\
1<r_{j}\leq q_{j},j=1,2,3,\\
\frac{2\beta-2}{d-\lambda}<\frac{1}{q_{1}}+\frac{1}{q_{2}}<\frac{3\beta-3}{d-\lambda},
\frac{1}{d-\lambda}<\frac{1}{q_{2}}+\frac{1}{q_{3}}<\frac{2\beta-1}{d-\lambda},\\
\frac{1}{q_{2}}-\frac{1}{q_{1}}<\frac{\beta-1}{d-\lambda},\frac{1}{q_{2}}-\frac{1}{q_{3}}<\frac{2\beta-3}{d-\lambda},
\end{cases}
\end{align}
and
\begin{align}\label{canshufanwei2}
\begin{cases}
\beta_{1}=\beta-1-\frac{d-\lambda}{r_{1}},\chi_{1}=2\alpha-\frac{2\alpha}{\beta}-\frac{2\alpha(d-\lambda)}{\beta q_{1}},\\
\beta_{2}=2\beta-2-\frac{d-\lambda}{r_{2}},\chi_{2}=4\alpha-\frac{4\alpha}{\beta}-\frac{2\alpha(d-\lambda)}{\beta q_{2}},\\
\beta_{3}=1-\frac{d-\lambda}{r_{3}},\chi_{3}=\frac{2\alpha}{\beta}-\frac{2\alpha(d-\lambda)}{\beta q_{3}}.
\end{cases}
\end{align}
Let $BC((0,\infty);X)$ denote the set of bounded functions from $(0,\infty)$ to $X$. Define a topological vector space $\mathbb{F}=\mathbb{Y}_{1}\times\mathbb{Y}_{2}\times\mathbb{Y}_{3}$ endowed with the product norm
\[
\big\|\left(u,v,w\right)\big\|_{\mathbb{F}}=\big\|u\big\|_{\mathbb{Y}_{1}}+\big\|v\big\|_{\mathbb{Y}_{2}}
+\big\|w\big\|_{\mathbb{Y}_{3}},
\]
where
\begin{align*}
\mathbb{Y}_{1}&=\left\{u\,|\,\text{ div } u=0,u\in BC\left((0,\infty);\dot{\mathcal{N}}^{-\beta_{1}}_{r_{1},\lambda,\infty}\right),
t^{\frac{\chi_{1}}{2}}u\in BC\left((0,\infty);\mathcal{M}_{q_{1},\lambda}\right)\right\},\\
\mathbb{Y}_{2}&=\left\{v\,|\,v\in BC\left((0,\infty);\dot{\mathcal{N}}^{-\beta_{2}}_{r_{2},\lambda,\infty}\right),
t^{\frac{\chi_{2}}{2}}v\in BC\left((0,\infty);\mathcal{M}_{q_{2},\lambda}\right)\right\},\\
\mathbb{Y}_{3}&=\left\{w\,|\,\nabla c\in BC\left((0,\infty);\dot{\mathcal{N}}^{-\beta_{3}}_{r_{3},\lambda,\infty}\right),
t^{\frac{\chi_{3}}{2}}\nabla w\in BC\left((0,\infty);\mathcal{M}_{q_{3},\lambda}\right),w\in BC\left((0,\infty);\mathcal{M}_{\frac{d-\lambda}{2-\beta},\lambda}^{2-\beta}\right)\right\},
\end{align*}
and the norms are defined as
\begin{align*}
\big\|u\big\|_{\mathbb{Y}_{1}}&=\sup_{t>0}\big\|u(t)\big\|_{\dot{\mathcal{N}}^{-\beta_{1}}_{r_{1},\lambda,\infty}}+
\sup_{t>0}\left(t^{\frac{\chi_{1}}{2}}\big\|u(t)\big\|_{\mathcal{M}_{q_{1},\lambda}}\right),\\
\big\|v\big\|_{\mathbb{Y}_{2}}&=\sup_{t>0}\big\|v(t)\big\|_{\dot{\mathcal{N}}^{-\beta_{2}}_{r_{2},\lambda,\infty}}+
\sup_{t>0}\left(t^{\frac{\chi_{2}}{2}}\big\|v(t)\big\|_{\mathcal{M}_{q_{2},\lambda}}\right),\\
\big\|w\big\|_{\mathbb{Y}_{3}}&=\sup_{t>0}\big\|\nabla w(t)\big\|_{\dot{\mathcal{N}}^{-\beta_{3}}_{r_{3},\lambda,\infty}}+
\sup_{t>0}\left(t^{\frac{\chi_{3}}{2}}\big\|\nabla w(t)\big\|_{\mathcal{M}_{q_{3},\lambda}}\right)+\sup_{t>0}\big\| w(t)\big\|_{\mathcal{M}_{\frac{d-\lambda}{2-\beta},\lambda}^{2-\beta}}.
\end{align*}

For $(u,v,w)\in\mathbb{F}$, we define map
$$
\Phi(u,v,w)=\big(\Phi_{1}(u,v,w),\Phi_{2}(u,v,w),\Phi_{3}(u,v,w)\big),
$$
where
\begin{align*}
\begin{cases}
\Phi_{1}(u,v,w) = \mathcal{S}_{\alpha,\beta}(t)u_{0} - \int_{0}^{t}(t-\theta)^{\alpha-1}\mathcal{P}_{\alpha,\beta}(t-\theta)\mathbb{P}(u\cdot\nabla u + n\nabla\phi) \, d\theta,\\
\Phi_{2}(u,v,w) = \mathcal{S}_{\alpha,\beta}(t)v_{0}- \int_{0}^{t}(t-\theta)^{\alpha-1}\mathcal{P}_{\alpha,\beta}(t-\theta)(u\cdot\nabla v + \nabla\cdot(v\cdot\nabla B(w))) \, d\theta,\\
\Phi_{3}(u,v,w) = \mathcal{S}_{\alpha,\beta}(t)w_{0} - \int_{0}^{t}(t-\theta)^{\alpha-1}\mathcal{P}_{\alpha,\beta}(t-\theta)(u\cdot\nabla w + ((-\Delta)^{\frac{2-\beta}{2}}w)v) \, d\theta,
\end{cases}
\end{align*}
It is also noteworthy that the Riesz transform $R_{j}$ is bounded on the Morrey space $\mathcal{M}_{p,\lambda}$ for $1<p<\infty$, as demonstrated by the Calder\'on-Zygmund operator. This fact can be found in the reference \cite{M. Taylor}. Consequently, the Leray projection operator is also bounded on $\mathcal{M}_{p,\lambda}$.
\begin{lemma}\label{selfmap}
Assuming conditions \eqref{canshufanwei1} and \eqref{canshufanwei2} are satisfied, and there exists a small positive constant $\kappa>0$ such that
\[
\big\|(u_{0},v_{0},\nabla w_{0})\big\|_{\dot{\mathcal{N}}^{-\beta_{1}}_{r_{1},\lambda,\infty}
\times\dot{\mathcal{N}}^{-\beta_{2}}_{r_{2},\lambda,\infty}\times\dot{\mathcal{N}}^{-\beta_{3}}_{r_{3},\lambda,\infty}}
+\big\|w_{0}\big\|_{\mathcal{M}^{2-\beta}_{\frac{d-\lambda}{2-\beta},\lambda}}+\big\|\nabla\phi\big\|_{d-\lambda,\lambda}\leq\kappa,
\]
for $(u,v,w)\in\mathbb{F}_{\kappa}$, it follows that $\Phi(u,v,w)\in\mathbb{F}_{\kappa}$, where
\[
\mathbb{F}_{\kappa}=\left\{(u,v,w)\in\mathbb{F}:\big\|u\big\|_{\mathbb{Y}_{1}}+\big\|v\big\|_{\mathbb{Y}_{2}}
+\big\|w\big\|_{\mathbb{Y}_{3}}\lesssim\kappa\right\}.
\]
\end{lemma}
\begin{proof}
\textbf{Step 1}. We estimate
\[
\left\|\int_{0}^{t}(t-\theta)^{\alpha-1}\mathcal{P}_{\alpha,\beta}(t-\theta)\mathbb{P}(\nabla u(s)\cdot u(s) + v(s)\nabla\phi) \, d\theta\right\|_{\mathbb{Y}_{1}}.
\]
Noting that $u(s)\cdot\nabla u(s)=\nabla\cdot\left(u(s)\otimes u(s)\right)$ and according to Proposition \ref{some propesitions} and Lemma \ref{Mittag estimats lemma},
\begin{align*}
&\left\|\int_{0}^{t}(t-\theta)^{\alpha-1}\mathcal{P}_{\alpha,\beta}(t-\theta)\mathbb{P}(\nabla u(\theta)\cdot u(\theta) + v(\theta)\nabla\phi) \, d\theta\right\|_{\dot{\mathcal{N}}^{-\beta_{1}}_{r_{1},\lambda,\infty}}\\ &\lesssim\left\|\int_{0}^{t}(t-\theta)^{\alpha-1}\mathcal{P}_{\alpha,\beta}(t-\theta)\mathbb{P}
(\nabla\cdot(u(\theta)\otimes u(\theta)) + v(\theta)\nabla\phi) \, d\theta\right\|_{\dot{\mathcal{N}}^{0}_{\frac{d-\lambda}{\beta-1},\lambda,\infty}}\\
&\lesssim\left\|\int_{0}^{t}(t-\theta)^{\alpha-1}\mathcal{P}_{\alpha,\beta}(t-\theta)\mathbb{P}
(\nabla\cdot(u(\theta)\otimes u(\theta)) + v(\theta)\nabla\phi) \, d\theta\right\|_{\mathcal{M}_{\frac{d-\lambda}{\beta-1},\lambda}}\\
&\lesssim\int_{0}^{t}(t-\theta)^{\alpha-1-\frac{\alpha}{\beta}-\frac{\alpha(d-\lambda)}{\beta}(\frac{2}{q_{1}}-\frac{\beta-1}{d-\lambda})}
\big\|u\otimes u\big\|_{\frac{q_{1}}{2},\lambda}d\theta\\
&\quad +(t-\theta)^{\alpha-1-\frac{\alpha(d-\lambda)}{\beta}(\frac{1}{q_{2}}+\frac{1}{d-\lambda}-\frac{\beta-1}{d-\lambda})}
\big\|v\nabla\phi\big\|_{\mathcal{M}_{\frac{q_{2}(d-\lambda)}{q_{2}+d-\lambda},\lambda}}d\theta\\
&\lesssim\int_{0}^{t}(t-\theta)^{2\alpha-\frac{2\alpha}{\beta}-\frac{2\alpha(d-\lambda)}{\beta q_{1}}-1}\big\|u\big\|_{q_{1},\lambda}^{2}ds+\int_{0}^{t}(t-\theta)^{2\alpha-\frac{2\alpha}{\beta}-\frac{\alpha(d-\lambda)}{\beta q_{2}}-1}\big\|v\big\|_{\mathcal{M}_{q_{2},\lambda}}\big\|\nabla\phi\big\|_{\mathcal{M}_{d-\lambda,\lambda}}d\theta\\
&\lesssim\int_{0}^{t}(t-\theta)^{2\alpha-\frac{2\alpha}{\beta}-\frac{2\alpha(d-\lambda)}{\beta q_{1}}-1}\theta^{-2\alpha+\frac{2\alpha}{\beta}+\frac{2\alpha(d-\lambda)}{\beta q_{1}}}d\theta\big[\sup_{t>0}\big(t^{\frac{\chi_{1}}{2}}\big\|u(t)\big\|_{\mathcal{M}_{q_{1},\lambda}}\big)\big]^{2}\\
&\quad +\int_{0}^{t}(t-\theta)^{2\alpha-\frac{2\alpha}{\beta}-\frac{\alpha(d-\lambda)}{\beta q_{2}}-1}\theta^{-2\alpha+\frac{2\alpha}{\beta}+\frac{\alpha(d-\lambda)}{\beta q_{2}}}d\theta\big(\sup_{t>0}t^{\frac{\chi_{2}}{2}}\big\|n(t)\big\|_{\mathcal{M}_{q_{2},\lambda}}\big)
\big\|\nabla\phi\big\|_{\mathcal{M}_{d-\lambda,\lambda}}\\
&\lesssim \big\|u\big\|_{\mathbb{Y}_{1}}^{2}+\big\|v\big\|_{\mathbb{Y}_{2}}\big\|\nabla\phi\big\|_{\mathcal{M}_{d-\lambda,\lambda}},
\end{align*}
here, we have employed the definition of the Beta function and, under conditions \eqref{canshufanwei1} and \eqref{canshufanwei2}, we notice that the inequalities $0 < 2\alpha - \frac{2\alpha}{\beta} - \frac{2\alpha(d-\lambda)}{\beta q_{1}} < 1$ and $0 < 2\alpha - \frac{2\alpha}{\beta} - \frac{\alpha(d-\lambda)}{\beta q_{2}} < 1$ hold.

Moreover,
\begin{align*}
&\left\|\int_{0}^{t}(t-\theta)^{\alpha-1}\mathcal{P}_{\alpha,\beta}(t-\theta)\mathbb{P}
(\nabla\cdot(u(\theta)\otimes u(\theta)) + v(\theta)\nabla\phi) \, d\theta\right\|_{\mathcal{M}_{q_{1},\lambda}}\\
&\lesssim\int_{0}^{t}(t-\theta)^{\alpha-1-\frac{\alpha}{\beta}-\frac{\alpha(d-\lambda)}{\beta}(\frac{2}{q_{1}}-\frac{1}{q_{1}})}\big\|u
\otimes u\big\|_{\frac{q_{1}}{2},\lambda}d\theta\\
&\quad +\int_{0}^{t}(t-\theta)^{\alpha-1-\frac{\alpha(d-\lambda)}{\beta}(\frac{1}{q_{2}}+\frac{1}{d-\lambda}-\frac{1}{q_{1}})
}\big\|v\big\|_{q_{2},\lambda}\big\|\nabla\phi\big\|_{d-\lambda,\lambda}d\theta\\
&\lesssim\int_{0}^{t}(t-\theta)^{\alpha-\frac{\alpha(d-\lambda)}{\beta q_{1}}-1}\theta^{-2\alpha+\frac{2\alpha}{\beta}+\frac{2\alpha(d-\lambda)}{\beta q_{1}}}d\theta\big[\sup_{t>0}t^{\frac{\chi_{1}}{2}}\big\|u(t)\big\|_{q_{1},\lambda}\big]^{2}\\
&\quad +\int_{0}^{t}(t-\theta)^{\alpha-\frac{\alpha}{\beta}-\frac{\alpha(d-\lambda)}{\beta q_{2}}+\frac{\alpha(d-\lambda)}{\beta q_{1}}-1}
\theta^{-2\alpha+\frac{2\alpha}{\beta}+\frac{\alpha(d-\lambda)}{\beta q_{2}}}d\theta\big(\sup_{t>0}t^{\frac{\chi_{2}}{2}}\big\|v(t)\big\|_{q_{2},\lambda}\big)\big\|\nabla\phi\big\|_{d-\lambda,\lambda}\\
&\lesssim t^{-\alpha+\frac{\alpha}{\beta}+\frac{\alpha(d-\lambda)}{\beta q_{1}}}\big\|u\big\|_{\mathbb{Y}_{1}}^{2}+t^{-\alpha+\frac{\alpha}{\beta}+\frac{\alpha(d-\lambda)}{\beta q_{1}}}\big\|v\big\|_{\mathbb{Y}_{2}}\big\|\nabla\phi\big\|_{d-\lambda,\lambda},
\end{align*}
where, we note that when conditions \eqref{canshufanwei1} are satisfied, we have that
$\alpha-\frac{\alpha}{\beta}-\frac{\alpha(d-\lambda)}{\beta q_{2}}+\frac{\alpha(d-\lambda)}{\beta q_{1}}>0$. Thus, from above estimate, we get
$$
\left\|\int_{0}^{t}(t-\theta)^{\alpha-1}\mathcal{P}_{\alpha,\beta}(t-\theta)\mathbb{P}(u\cdot\nabla u + v\nabla\phi) \, d\theta\right\|_{\mathbb{Y}_{1}}\lesssim \big\|u\big\|_{\mathbb{Y}_{1}}^{2}+\big\|v\big\|_{\mathbb{Y}_{2}}\big\|\nabla\phi\big\|_{d-\lambda,\lambda}.
$$

\textbf{Step 2}. We estimate
$$
\left\|\int_{0}^{t}(t-\theta)^{\alpha-1}\mathcal{P}_{\alpha,\beta}(t-\theta)\big(u\cdot\nabla v + \nabla\cdot\big(v\nabla(B(w)\big)\big) \, d\theta\right\|_{\mathbb{Y}_{2}}.
$$
Using Proposition \ref{some propesitions}, Lemma \ref{Mittag estimats lemma},and note that $u\cdot \nabla v=\nabla\cdot\big(u\otimes v\big)$ and denote $\frac{1}{\chi}=\frac{1}{q_{3}}-\frac{d-\lambda}{2-\beta}$, we get
\begin{align*}
&\left\|\int_{0}^{t}(t-\theta)^{\alpha-1}\mathcal{P}_{\alpha,\beta}(t-\theta)\bigg(u\cdot\nabla v + \nabla\cdot\big(v\nabla(B(w)\big)\bigg) \, d\theta\right\|_{\dot{\mathcal{N}}^{-\beta_{2}}_{r_{2},\lambda,\infty}}\\
&\lesssim \left\|\int_{0}^{t}(t-\theta)^{\alpha-1}\mathcal{P}_{\alpha,\beta}(t-\theta)\bigg(\nabla\cdot\big(u\otimes v\big) + \nabla\cdot\big(v\nabla(B(w)\big)\bigg) \, d\theta\right\|_{\dot{\mathcal{N}}^{0}_{\frac{d-\lambda}{2\beta-2},\lambda,\infty}}\\
&\lesssim\left\|\int_{0}^{t}(t-\theta)^{\alpha-1}\mathcal{P}_{\alpha,\beta}(t-\theta)\bigg(\nabla\cdot
\big(u\otimes v\big) + \nabla\cdot\big(v\nabla(B(w)\big)\bigg) \, d\theta\right\|_{\mathcal{M}_{\frac{d-\lambda}{2\beta-2},\lambda}}\\
&\lesssim\int_{0}^{t}(t-\theta)^{\alpha-\frac{\alpha}{\beta}-\frac{\alpha(d-\lambda)}{\beta}(\frac{1}{q_{1}}+\frac{1}{q_{2}}
-\frac{2\beta-2}{d-\lambda})-1}\big\|u\otimes v\big\|_{\frac{q_{1}q_{2}}{q_{1}+q_{2}},\lambda}d\theta\\
&\quad +\int_{0}^{t}(t-\theta)^{\alpha-\frac{\alpha}{\beta}-\frac{\alpha(d-\lambda)}{\beta}(\frac{1}{q_{2}}+\frac{1}
{\chi}-\frac{2\beta-2}{d-\lambda})-1}\big\|v\nabla \big((-\Delta)^{\frac{\beta}{2}-1}w\big)\big\|_{\frac{q_{2}\chi}{q_{2}+\chi},\lambda}d\theta\\
&\lesssim\int_{0}^{t}(t-\theta)^{3\alpha-\frac{3\alpha}{\beta}-\frac{\alpha (d-\lambda)}{\beta q_{1}}-\frac{\alpha (d-\lambda)}{\beta q_{2}}-1}\big\|u\big\|_{q_{1},\lambda}\big\|v\big\|_{q_{2},\lambda}d\theta\\
&\quad +\int_{0}^{t}(t-\theta)^{2\alpha-\frac{\alpha}{\beta}-\frac{\alpha(d-\lambda)}{\beta q_{2}}-\frac{\alpha(d-\lambda)}{\beta q_{3}}
-1}\big\|v\big\|_{q_{2},\lambda}\big\|\nabla(-\Delta)^{\frac{\beta}{2}-1}w\big\|_{\chi,\lambda}d\theta\\
&\lesssim \int_{0}^{t}(t-\theta)^{3\alpha-\frac{3\alpha}{\beta}-\frac{\alpha (d-\lambda)}{\beta q_{1}}-\frac{\alpha (d-\lambda)}{\beta q_{2}}-1}\theta^{-3\alpha+\frac{3\alpha}{\beta}+\frac{\alpha (d-\lambda)}{\beta q_{1}}+\frac{\alpha (d-\lambda)}{\beta q_{2}}}d\theta\bigg(\sup_{t>0}t^{\frac{\chi_{1}}{2}}\big\|u\big\|_{q_{1},\lambda}\bigg)
\bigg(\sup_{t>0}t^{\frac{\chi_{2}}{2}}\big\|v\big\|_{q_{2},\lambda}\bigg)\\
&\quad +\int_{0}^{t}(t-\theta)^{2\alpha-\frac{\alpha}{\beta}-\frac{\alpha(d-\lambda)}{\beta q_{2}}-\frac{\alpha(d-\lambda)}{\beta q_{3}}
-1}\theta^{-2\alpha+\frac{\alpha}{\beta}+\frac{\alpha(d-\lambda)}{\beta q_{2}}+\frac{\alpha(d-\lambda)}{\beta q_{3}}}d\theta\bigg(\sup_{t>0}t^{\frac{\chi_{2}}{2}}\big\|v\big\|_{q_{2},\lambda}\bigg)
\bigg(\sup_{t>0}t^{\frac{\chi_{3}}{2}}\big\|\nabla w\big\|_{q_{3},\lambda}\bigg)\\
&\lesssim \big\|u\big\|_{\mathbb{Y}_{1}}\big\|v\big\|_{\mathbb{Y}_{2}}+\big\|v\big\|_{\mathbb{Y}_{2}}\big\|w\big\|_{\mathbb{Y}_{3}},
\end{align*}
where, it is observed that $0 < 2\alpha - \frac{\alpha}{\beta} - \frac{\alpha(d-\lambda)}{\beta q_{2}} - \frac{\alpha(d-\lambda)}{\beta q_{3}} < 1$ and $0 < 3\alpha - \frac{3\alpha}{\beta} - \frac{\alpha (d-\lambda)}{\beta q_{1}} - \frac{\alpha (d-\lambda)}{\beta q_{2}} < 1$ hold under the condition \eqref{canshufanwei1}.

Moreover,
\begin{align*}
&\left\|\int_{0}^{t}(t-\theta)^{\alpha-1}\mathcal{P}_{\alpha,\beta}(t-\theta)
\bigg(\nabla\cdot\big(u\otimes v\big) + \nabla\cdot\big(v\nabla(B(w)\big)\bigg) \, d\theta\right\|_{\mathcal{M}_{q_{2},\lambda}(\mathbb{R}^{d})}\\
&\lesssim\int_{0}^{t}(t-\theta)^{\alpha-\frac{\alpha}{\beta}-\frac{\alpha(d-\lambda)}{\beta}(\frac{1}{q_{1}}+\frac{1}{q_{2}}-
\frac{1}{q_{2}})-1}\big\|u\big\|_{\mathcal{M}_{q_{1},\lambda}(\mathbb{R}^{d})}
\big\|v\big\|_{\mathcal{M}_{q_{2},\lambda}(\mathbb{R}^{d})}d\theta\\
&\quad +\int_{0}^{t}(t-\theta)^{\alpha-\frac{\alpha}{\beta}-\frac{\alpha(d-\lambda)}{\beta}(\frac{1}{q_{2}}+\frac{1}{\chi}-\frac{1}{q_{2}})-1}
\big\|v\big\|_{\mathcal{M}_{q_{2},\lambda}(\mathbb{R}^{d})}
\big\|\nabla(-\Delta)^{\frac{\beta}{2}-1}w\big\|_{\mathcal{M}_{\chi,\lambda}(\mathbb{R}^{d})}d\theta\\
&\lesssim\int_{0}^{t}(t-\theta)^{\alpha-\frac{\alpha}{\beta}-\frac{\alpha(d-\lambda)}{\beta q_{1}}-1}\theta^{-3\alpha+\frac{3\alpha}{\beta}+\frac{\alpha(d-\lambda)}{\beta q_{1}}+\frac{\alpha(d-\lambda)}{\beta q_{2}}}d\theta\bigg(\sup_{t>0}t^{\frac{\chi_{1}}{2}}\big\|u(t)\big\|_{\mathcal{M}_{q_{1},\lambda}(\mathbb{R}^{d})}\bigg)
\bigg(\sup_{t>0}t^{\frac{\chi_{2}}{2}}\big\|v(t)\big\|_{\mathcal{M}_{q_{2},\lambda}(\mathbb{R}^{d})}\bigg)\\
&\quad +\int_{0}^{t}(t-\theta)^{\frac{\alpha}{\beta}-\frac{\alpha(d-\lambda)}{\beta q_{3}}-1}\theta^{-2\alpha+\frac{\alpha}{\beta}+\frac{\alpha(d-\lambda)}{\beta q_{2}}+\frac{\alpha(d-\lambda)}{\beta q_{3}}}d\theta\bigg(\sup_{t>0}t^{\frac{\chi_{2}}{2}}\big\|v(t)\big\|_{\mathcal{M}_{q_{2},\lambda}(\mathbb{R}^{d})}\bigg)
\bigg(\sup_{t>0}t^{\frac{\chi_{3}}{2}}\big\|\nabla w(t)\big\|_{\mathcal{M}_{q_{3},\lambda}(\mathbb{R}^{d})}\bigg)\\
&\lesssim t^{-2\alpha+\frac{2\alpha}{\beta}+\frac{\alpha(d-\lambda)}{\beta q_{2}}}\big\|u\big\|_{\mathbb{Y}_{1}}\big\|v\big\|_{\mathbb{Y}_{2}}+t^{-2\alpha+\frac{2\alpha}{\beta}+\frac{\alpha(d-\lambda)}{\beta q_{2}}}\big\|v\big\|_{\mathbb{Y}_{2}}\big\|w\big\|_{\mathbb{Y}_{3}}.
\end{align*}
Therefore, from the above estimate, we get that
\begin{align*}
&\left\|\int_{0}^{t}(t-\theta)^{\alpha-1}\mathcal{P}_{\alpha,\beta}(t-\theta)\bigg(u\cdot\nabla v + \nabla\cdot\big(v\nabla(B(w)\big)\bigg) \, d\theta\right\|_{\mathbb{Y}_{2}}\\
&\lesssim\big\|u\big\|_{\mathbb{Y}_{1}}\big\|v\big\|_{\mathbb{Y}_{2}}+\big\|v\big\|_{\mathbb{Y}_{2}}\big\|w\big\|_{\mathbb{Y}_{3}}.
\end{align*}

\textbf{Step 3}. We estimate
$$
\left\|\int_{0}^{t}(t-\theta)^{\alpha-1}\mathcal{P}_{\alpha,\beta}(t-\theta)\bigg(u\cdot\nabla w + \big((-\Delta)^{\frac{2-\beta}{2}}w\big)v\bigg) \, d\theta\right\|_{\mathbb{Y}_{3}}.
$$
Using Proposition \ref{some propesitions}, Lemma \ref{Mittag estimats lemma}, we get
\begin{align*}
&\left\|\nabla\int_{0}^{t}(t-\theta)^{\alpha-1}\mathcal{P}_{\alpha,\beta}(t-\theta)
\bigg(u\cdot\nabla w + \big((-\Delta)^{\frac{2-\beta}{2}}w\big)v\bigg) \, d\theta\right\|_{\dot{\mathcal{N}}^{-\beta_{3}}_{r_{3},\lambda,\infty}}\\
&\lesssim\left\|\nabla\int_{0}^{t}(t-\theta)^{\alpha-1}\mathcal{P}_{\alpha,\beta}(t-\theta)
\bigg(u\cdot\nabla w + \big((-\Delta)^{\frac{2-\beta}{2}}w\big)v\bigg) \, d\theta\right\|_{\dot{\mathcal{N}}^{0}_{d-\lambda,\lambda,\infty}}\\
&\lesssim\left\|\nabla\int_{0}^{t}(t-\theta)^{\alpha-1}\mathcal{P}_{\alpha,\beta}(t-\theta)
\bigg(u\cdot\nabla w + \big((-\Delta)^{\frac{2-\beta}{2}}w\big)v\bigg) \, d\theta\right\|_{\mathcal{M}_{d-\lambda,\lambda}}\\
&\lesssim\int_{0}^{t}(t-\theta)^{\alpha-\frac{\alpha}{\beta}-\frac{\alpha(d-\lambda)}{\beta}(\frac{1}{q_{1}}+
\frac{1}{q_{3}}-\frac{1}{d-\lambda})-1}
\big\|u\big\|_{q_{1},\lambda}\big\|\nabla w\big\|_{q_{3},\lambda}d\theta\\
&\quad +\int_{0}^{t}(t-\theta)^{\alpha-\frac{\alpha}{\beta}-\frac{\alpha(d-\lambda)}{\beta}(\frac{1}{q_{2}}+\frac{2-\beta}{d-\lambda}
-\frac{1}{d-\lambda})-1}\big\|v\big\|_{q_{2},\lambda}\big\|w\big\|_{\mathcal{M}^{2-\beta}_{\frac{d-\lambda}{2-\beta},\lambda}}d\theta\\
&\lesssim\int_{0}^{t}(t-\theta)^{\alpha-\frac{\alpha(d-\lambda)}{\beta q_{1}}-\frac{\alpha(d-\lambda)}{\beta q_{3}}-1}\theta^{-\alpha+\frac{\alpha(d-\lambda)}{\beta q_{1}}+\frac{\alpha(d-\lambda)}{\beta q_{3}}}d\theta
\bigg(\sup_{t>0}t^{\frac{\chi_{3}}{2}}\big\|\nabla w(t)\big\|_{\mathcal{M}_{q_{3},\lambda}}\bigg)\bigg(\sup_{t>0}t^{\frac{\chi_{1}}{2}}
\big\|u(t)\big\|_{\mathcal{M}_{q_{1},\lambda}}\bigg)\\
&\quad +\int_{0}^{t}(t-\theta)^{2\alpha-\frac{2\alpha}{\beta}-\frac{\alpha(d-\lambda)}{\beta q_{2}}-1}\theta^{-2\alpha+\frac{2\alpha}{\beta}+\frac{\alpha(d-\lambda)}{\beta q_{2}}}d\theta\bigg(\sup_{t>0}t^{\frac{\chi_{2}}{2}}
\big\|v(t)\big\|_{\mathcal{M}_{q_{2},\lambda}}\bigg)\bigg(\sup_{t>0}\big\|w(t)\big\|
_{\mathcal{M}^{2-\beta}_{\frac{d-\lambda}{2-\beta},\lambda}}\bigg)\\
&\lesssim \big\|u\big\|_{\mathbb{Y}_{1}}\big\|w\big\|_{\mathbb{Y}_{3}}+\big\|v\big\|_{\mathbb{Y}_{2}}
\big\|w\big\|_{\mathbb{Y}_{3}},
\end{align*}
where under the condition that \eqref{canshufanwei1} is satisfied, we obtain that
$0<\alpha-\frac{\alpha(d-\lambda)}{\beta q_{1}}-\frac{\alpha(d-\lambda)}{\beta q_{3}}<1$.

Moreover,
\begin{align*}
&\left\|\nabla\int_{0}^{t}(t-\theta)^{\alpha-1}\mathcal{P}_{\alpha,\beta}(t-\theta)
\bigg(u\cdot\nabla w + \big((-\Delta)^{\frac{2-\beta}{2}}w\big)v\bigg) \, d\theta\right\|_{\mathcal{M}_{q_{3},\lambda}}\\
&\lesssim \int_{0}^{t}(t-\theta)^{\alpha-\frac{\alpha}{\beta}-\frac{\alpha(d-\lambda)}{\beta}(\frac{1}{q_{1}}+\frac{1}{q_{3}}-\frac{1}{q_{3}})-1}
\big\|u\big\|_{q_{1},\lambda}\big\|\nabla w\big\|_{q_{3},\lambda}d\theta\\
&\quad +\int_{0}^{t}(t-\theta)^{\alpha-\frac{\alpha}{\beta}-\frac{\alpha(d-\lambda)}{\beta}(\frac{1}{q_{2}}
+\frac{2-\beta}{d-\lambda}-\frac{1}{q_{3}})-1}\big\|(-\Delta)^{\frac{2-\beta}{2}}w\big)v\big\|_{\frac{q_{2}(d-\lambda)}{d-\lambda
+(2-\beta)q_{2}},\lambda}d\theta\\
&\lesssim \int_{0}^{t}(t-\theta)^{\alpha-\frac{\alpha}{\beta}-\frac{\alpha(d-\lambda)}{\beta q_{1}}-1}\theta^{-\alpha+\frac{\alpha(d-\lambda)}{\beta q_{1}}+\frac{\alpha(d-\lambda)}{\beta q_{3}}}d\theta\bigg(\sup_{t>0}t^{\frac{\chi_{3}}{2}}\big\|\nabla w(t)\big\|_{\mathcal{M}_{q_{3},\lambda}(\mathbb{R}^{d})}\bigg)\bigg(\sup_{t>0}t^{\frac{\chi_{1}}{2}}
\big\|u(t)\big\|_{\mathcal{M}_{q_{1},\lambda}(\mathbb{R}^{d})}\bigg)\\
&\quad +\int_{0}^{t}(t-\theta)^{2\alpha-\frac{3\alpha}{\beta}-\frac{\alpha(d-\lambda)}{\beta q_{2}}
+\frac{\alpha(d-\lambda)}{\beta q_{3}}-1}\theta^{-2\alpha+\frac{2\alpha}{\beta}+\frac{\alpha(d-\lambda)}{\beta q_{2}}}d\theta\bigg(\sup_{t>0}\big\|w(t)\big\|_{\mathcal{M}^{2-\beta}_{\frac{d-\lambda}{2-\beta},\lambda}}\bigg)
\bigg(\sup_{t>0}t^{\frac{\chi_{2}}{2}}\big\|v(t)\big\|_{\mathcal{M}_{q_{1},\lambda}(\mathbb{R}^{d})}\bigg)\\
&\lesssim t^{-\frac{\alpha}{\beta}+\frac{\alpha(d-\lambda)}{\beta q_{3}}}\big\|u\big\|_{\mathbb{Y}_{1}}\big\|w\big\|_{\mathbb{Y}_{3}}
+t^{-\frac{\alpha}{\beta}+\frac{\alpha(d-\lambda)}{\beta q_{3}}}\big\|v\big\|_{\mathbb{Y}_{2}}\big\|w\big\|_{\mathbb{Y}_{3}},
\end{align*}
where, we observe that when \eqref{canshufanwei1} is satisfied, we have $2\alpha - \frac{3\alpha}{\beta} - \frac{\alpha(d-\lambda)}{\beta q_{2}} + \frac{\alpha(d-\lambda)}{\beta q_{3}} > 0$.
As well as,
\begin{align*}
&\left\|\int_{0}^{t}(t-\theta)^{\alpha-1}\mathcal{P}_{\alpha,\beta}(t-\theta)
\bigg(u\cdot\nabla w + \big((-\Delta)^{\frac{2-\beta}{2}}w\big)v\bigg) \, d\theta\right\|_{\mathcal{M}_{\frac{d-\lambda}{2-\beta},\lambda}^{2-\beta}}\\
&\lesssim\left\|(-\Delta)^{\frac{2-\beta}{2}}\int_{0}^{t}(t-\theta)^{\alpha-1}\mathcal{P}_{\alpha,\beta}(t-\theta)
\bigg(u\cdot\nabla w + \big((-\Delta)^{\frac{2-\beta}{2}}w\big)v\bigg) \, d\theta\right\|_{\mathcal{M}_{\frac{d-\lambda}{2-\beta},\lambda}}\\
&\lesssim\int_{0}^{t}(t-\theta)^{\alpha-\frac{\alpha(2-\beta)}{\beta}-\frac{\alpha(d-\lambda)}{\beta}(\frac{1}{q_{1}}+\frac{1}{q_{3}}
-\frac{2-\beta}{d-\lambda})-1}\big\|u\big\|_{q_{1},\lambda}\big\|\nabla w\big\|_{q_{3},\lambda}d\theta\\
&\quad +\int_{0}^{t}(t-\theta)^{\alpha-\frac{\alpha(2-\beta)}{\beta}-\frac{\alpha(d-\lambda)}{\beta}(\frac{1}{q_{1}}+\frac{2-\beta}{d-\lambda}
-\frac{2-\beta}{d-\lambda})-1}\big\|v\big\|_{q_{2},\lambda}\big\|w\big\|
_{\mathcal{M}^{2-\beta}_{\frac{d-\lambda}{2-\beta},\lambda}}d\theta\\
&\lesssim \int_{0}^{t}(t-\theta)^{\alpha-\frac{\alpha(d-\lambda)}{\beta q_{1}}-\frac{\alpha(d-\lambda)}{\beta q_{3}}-1}\theta^{-\alpha+\frac{\alpha(d-\lambda)}{\beta q_{1}}+\frac{\alpha(d-\lambda)}{\beta q_{3}}}d\theta\bigg(\sup_{t>0}t^{\frac{\chi_{3}}{2}}\big\|\nabla w(t)\big\|_{\mathcal{M}_{q_{3},\lambda}(\mathbb{R}^{d})}\bigg)\bigg(\sup_{t>0}t^{\frac{\chi_{1}}{2}}\big\|u(t)\big\|
_{\mathcal{M}_{q_{1},\lambda}(\mathbb{R}^{d})}\bigg)\\
&\quad +\int_{0}^{t}(t-\theta)^{2\alpha-\frac{2\alpha}{\beta}-\frac{\alpha(d-\lambda)}{\beta q_{2}}-1}\theta^{-2\alpha+\frac{2\alpha}{\beta}+\frac{\alpha(d-\lambda)}{\beta q_{2}}}d\theta
\bigg(\sup_{t>0}\big\|w(t)\big\|_{\mathcal{M}^{2-\beta}_{\frac{d-\lambda}{2-\beta},\lambda}}\bigg)
\bigg(\sup_{t>0}t^{\frac{\chi_{2}}{2}}\big\|v(t)\big\|_{\mathcal{M}_{q_{2},\lambda}(\mathbb{R}^{d})}\bigg)\\
&\lesssim \big\|u\big\|_{\mathbb{Y}_{1}}\big\|w\big\|_{\mathbb{Y}_{3}}+\big\|v\big\|_{\mathbb{Y}_{2}}\big\|w\big\|_{\mathbb{Y}_{3}},
\end{align*}
Therefore£¬ from above estimate, we get
\begin{align*}
&\left\|\int_{0}^{t}(t-\theta)^{\alpha-1}\mathcal{P}_{\alpha,\beta}(t-\theta)\bigg(u\cdot\nabla w + \big((-\Delta)^{\frac{2-\beta}{2}}w\big)v\bigg) \, d\theta\right\|_{\mathbb{Y}_{3}}\\
&\lesssim \big\|u\big\|_{\mathbb{Y}_{1}}\big\|w\big\|_{\mathbb{Y}_{3}}+\big\|v\big\|_{\mathbb{Y}_{2}}\big\|w\big\|_{\mathbb{Y}_{3}}.
\end{align*}

\textbf{Step 4.} We will estimate
\begin{align*}
\left\|\mathcal{S}_{\alpha,\beta}(t)u_{0}\right\|_{\mathbb{Y}_{1}},
\left\|\mathcal{S}_{\alpha,\beta}(t)v_{0}\right\|_{\mathbb{Y}_{2}},
\left\|\mathcal{S}_{\alpha,\beta}(t)w_{0}\right\|_{\mathbb{Y}_{3}}.
\end{align*}
respectively.
Using Lemma \ref{Mittag estimats lemma}, Propesition \ref{some propesitions}, we get
\begin{align*}
&\left\|\mathcal{S}_{\alpha,\beta}(t)u_{0}\right\|_{\mathbb{Y}_{1}}\\
&\lesssim \sup_{t>0}\left\|\mathcal{S}_{\alpha,\beta}(t)u_{0}\right\|_{\dot{\mathcal{N}}^{-\beta_{1}}_{r_{1},
\lambda,\infty}}+\sup_{t>0}t^{\frac{\chi_{1}}{2}}
\left\|\mathcal{S}_{\alpha,\beta}(t)u_{0}\right\|_{\mathcal{M}_{q_{1},\lambda}}\\
&\lesssim \big\|u_{0}\big\|_{\dot{\mathcal{N}}^{-\beta_{1}}_{r_{1},
\lambda,\infty}}+\sup_{t>0}t^{\frac{\chi_{1}}{2}}
\left\|\mathcal{S}_{\alpha,\beta}(t)u_{0}\right\|_{\dot{\mathcal{N}}^{0}_{q_{1},
\lambda,1}}\\
&\lesssim
\big\|u_{0}\big\|_{\dot{\mathcal{N}}^{-\beta_{1}}_{r_{1},
\lambda,\infty}}+\sup_{t>0}t^{\frac{\chi_{1}}{2}}t^{-\frac{\alpha}{\beta}(\beta-1-\frac{d-\lambda}{q_{1}})}
\left\|u_{0}\right\|_{\dot{\mathcal{N}}
^{-\beta+1+\frac{d-\lambda}{q_{1}}}_{q_{1},
\lambda,\infty}}\lesssim\big\|u_{0}\big\|_{\dot{\mathcal{N}}^{-\beta_{1}}_{r_{1},
\lambda,\infty}},
\end{align*}
\begin{align*}
&\left\|\mathcal{S}_{\alpha,\beta}(t)v_{0}\right\|_{\mathbb{Y}_{2}}\\
&\lesssim \sup_{t>0}\left\|\mathcal{S}_{\alpha,\beta}(t)v_{0}\right\|_{\dot{\mathcal{N}}^{-\beta_{2}}_{r_{2},
\lambda,\infty}}+\sup_{t>0}t^{\frac{\chi_{2}}{2}}
\left\|\mathcal{S}_{\alpha,\beta}(t)v_{0}\right\|_{\mathcal{M}_{q_{2},\lambda}}\\
&\lesssim \big\|v_{0}\big\|_{\dot{\mathcal{N}}^{-\beta_{2}}_{r_{2},
\lambda,\infty}}+\sup_{t>0}t^{\frac{\chi_{2}}{2}}
\left\|\mathcal{S}_{\alpha,\beta}(t)v_{0}\right\|_{\dot{\mathcal{N}}^{0}_{q_{2},
\lambda,1}}\\
&\lesssim
\big\|v_{0}\big\|_{\dot{\mathcal{N}}^{-\beta_{2}}_{r_{2},
\lambda,\infty}}+\sup_{t>0}t^{\frac{\chi_{2}}{2}}t^{-\frac{\alpha}{\beta}(2\beta-2-\frac{d-\lambda}{q_{2}})}
\left\|v_{0}\right\|_{\dot{\mathcal{N}}
^{-2\beta+2+\frac{d-\lambda}{q_{2}}}_{q_{2},
\lambda,\infty}}\lesssim\big\|v_{0}\big\|_{\dot{\mathcal{N}}^{-\beta_{2}}_{r_{2},
\lambda,\infty}},
\end{align*}
and
\begin{align*}
&\left\|\mathcal{S}_{\alpha,\beta}(t)w_{0}\right\|_{\mathbb{Y}_{3}}\\
&\lesssim \sup_{t>0}\left\|\nabla \mathcal{S}_{\alpha,\beta}(t)w_{0}\right\|_{\dot{\mathcal{N}}^{-\beta_{3}}_{r_{3},
\lambda,\infty}}+\sup_{t>0}t^{\frac{\chi_{3}}{2}}
\left\|\nabla \mathcal{S}_{\alpha,\beta}(t)w_{0}\right\|_{\mathcal{M}_{q_{3},\lambda}}\\
&\quad+\sup_{t>0}\left\|(-\Delta)^{\frac{2-\beta}{2}}\mathcal{S}_{\alpha,\beta}(t)w_{0}
\right\|_{\mathcal{M}_{\frac{d-\lambda}{2-\beta},\lambda}}\\
&\lesssim \big\|\nabla w_{0}\big\|_{\dot{\mathcal{N}}^{-\beta_{3}}_{r_{3},
\lambda,\infty}}+\sup_{t>0}t^{\frac{\chi_{3}}{2}}
\left\|\nabla \mathcal{S}_{\alpha,\beta}(t)w_{0}\right\|_{\dot{\mathcal{N}}^{0}_{q_{3},
\lambda,1}}+\big\|w_{0}\big\|_{\mathcal{M}^{2-\beta}_{\frac{d-\lambda}{2-\beta},\lambda}}\\
&\lesssim
\big\|\nabla w_{0}\big\|_{\dot{\mathcal{N}}^{-\beta_{3}}_{r_{3},
\lambda,\infty}}+\sup_{t>0}t^{\frac{\chi_{3}}{2}}t^{-\frac{\alpha}{\beta}(1-\frac{d-\lambda}{q_{3}})}
\left\|\nabla w_{0}\right\|_{\dot{\mathcal{N}}
^{-1+\frac{d-\lambda}{q_{3}}}_{q_{3},
\lambda,\infty}}+\big\|w_{0}\big\|_{\mathcal{M}^{2-\beta}_{\frac{d-\lambda}{2-\beta},\lambda}}\\
&\lesssim\big\|\nabla w_{0}\big\|_{\dot{\mathcal{N}}^{-\beta_{3}}_{r_{3},
\lambda,\infty}}+\big\|w_{0}\big\|_{\mathcal{M}^{2-\beta}_{\frac{d-\lambda}{2-\beta},\lambda}}.
\end{align*}

\textbf{Step 5}. From above estimates, when $\kappa$ is small, we get
\begin{align*}
&\big\|\Phi(u,v,w)\big\|_{\mathbb{F}}\\
&\lesssim\big\|\Phi_{1}(u,v,w)\big\|_{\mathbb{Y}_{1}}+\big\|\Phi_{2}(u,v,w)\big\|_{\mathbb{Y}_{2}}
+\big\|\Phi_{3}(u,v,w)\big\|_{\mathbb{Y}_{3}}\\
&\lesssim \big\|u\big\|_{\mathbb{Y}_{1}}^{^{2}}+\big\|v\big\|_{\mathbb{Y}_{2}}\big\|\nabla\phi\big\|_{d-\lambda,\lambda}+
\big\|u\big\|_{\mathbb{Y}_{1}}\big\|v\big\|_{\mathbb{Y}_{2}}+\big\|v\big\|_{\mathbb{Y}_{2}}\big\|w\big\|_{\mathbb{Y}_{3}}
+\big\|u\big\|_{\mathbb{Y}_{1}}\big\|w\big\|_{\mathbb{Y}_{3}}+\big\|v\big\|_{\mathbb{Y}_{2}}\big\|w\big\|_{\mathbb{Y}_{3}}\\
&\quad +\big\|u_{0}\big\|_{\dot{\mathcal{N}}^{-\beta_{1}}_{r_{1},
\lambda,\infty}}+\big\|v_{0}\big\|_{\dot{\mathcal{N}}^{-\beta_{2}}_{r_{2},
\lambda,\infty}}+\big\|\nabla w_{0}\big\|_{\dot{\mathcal{N}}^{-\beta_{3}}_{r_{3},
\lambda,\infty}}+\big\|w_{0}\big\|_{\mathcal{M}^{2-\beta}_{\frac{d-\lambda}{2-\beta},\lambda}}\\
&\lesssim \bigg(\big\|u\big\|_{\mathbb{Y}_{1}}+\big\|v\big\|_{\mathbb{Y}_{2}}+\big\|w\big\|_{\mathbb{Y}_{3}}\bigg)
\bigg(\big\|u\big\|_{\mathbb{Y}_{1}}+\big\|v\big\|_{\mathbb{Y}_{2}}+\big\|w\big\|_{\mathbb{Y}_{3}}+\big\|\nabla\phi\big\|
_{d-\lambda,\lambda}\bigg)\\
&\quad +\big\|u_{0}\big\|_{\dot{\mathcal{N}}^{-\beta_{1}}_{r_{1},
\lambda,\infty}}+\big\|v_{0}\big\|_{\dot{\mathcal{N}}^{-\beta_{2}}_{r_{2},
\lambda,\infty}}+\big\|\nabla w_{0}\big\|_{\dot{\mathcal{N}}^{-\beta_{3}}_{r_{3},
\lambda,\infty}}+\big\|w_{0}\big\|_{\mathcal{M}^{2-\beta}_{\frac{d-\lambda}{2-\beta},\lambda}}\lesssim\kappa.
\end{align*}
The proof is thus completed.
\end{proof}
\begin{lemma}\label{contraction operator}
Given the conditions of Lemma \ref{selfmap}, for any $(u_{1},v_{1},w_{1}),(u_{2},v_{2},w_{2}) \in \mathbb{F}_{\kappa}$ and
$$
\big(u_{1},v_{1},w_{1}\big)\big|_{t=0} = \big(u_{2},v_{2},w_{2}\big)\big|_{t=0} = \big(u_{0},v_{0},w_{0}\big),
$$
then the operator $\Phi = \left(\Phi_{1}, \Phi_{2}, \Phi_{3}\right)$ is a contraction operator.
\end{lemma}
\begin{proof}
For convenience, we denote that
$$
\big(\tilde{u},\tilde{v},\tilde{w}\big)=\big(u_{1}-u_{2},v_{1}-v_{2},w_{1}-w_{2}\big).
$$
Note that
\begin{align*}
&\bigg|\Phi_{1}(u_{1},v_{1},w_{1})-\Phi_{1}(u_{2},v_{2},w_{2})\bigg|\\
&=\bigg|\int_{0}^{t}(t-\theta)^{\alpha-1}\mathcal{P}_{\alpha,\beta}(t-\theta)\mathbb{P}
\left(u_{1}\cdot\nabla u_{1} + v_{1}\nabla\phi\right) \, d\theta\\
&\quad -\int_{0}^{t}(t-\theta)^{\alpha-1}\mathcal{P}_{\alpha,\beta}(t-\theta)\mathbb{P}
\left(u_{2}\cdot\nabla u_{2} + v_{2}\nabla\phi\right) \, d\theta\bigg|\\
&=\bigg|\int_{0}^{t}(t-\theta)^{\alpha-1}\mathcal{P}_{\alpha,\beta}(t-\theta)\mathbb{P}
\left(\tilde{u}\cdot\nabla u_{1}+u_{2}\nabla\tilde{u} + \tilde{v}\nabla\phi\right) \, d\theta\bigg|,
\end{align*}
\begin{align*}
&\bigg|\Phi_{2}(u_{1},v_{1},w_{1})-\Phi_{2}(u_{2},v_{2},w_{2})\bigg|\\
&=\bigg|\int_{0}^{t}(t-\theta)^{\alpha-1}\mathcal{P}_{\alpha,\beta}(t-\theta)\bigg(u_{1}\cdot\nabla v_{1} + \nabla\cdot\big(v_{1}\nabla(B(w_{1})\big)\bigg) \, d\theta\\
&\quad-\int_{0}^{t}(t-\theta)^{\alpha-1}\mathcal{P}_{\alpha,\beta}(t-\theta)\bigg(u_{2}\cdot\nabla v_{2} + \nabla\cdot\big(v_{2}\nabla(B(w_{2})\big)\bigg) \, d\theta\bigg|\\
&=\bigg|\int_{0}^{t}(t-\theta)^{\alpha-1}\mathcal{P}_{\alpha,\beta}(t-\theta)\bigg(\tilde{u}\cdot\nabla v_{1}+u_{2}\cdot\nabla \tilde{v} + \nabla\cdot\big(\tilde{v}\nabla(B(w_{1}))+v_{2}\nabla(B(\tilde{w})\big)\bigg) \, d\theta\bigg|,
\end{align*}
and
\begin{align*}
&\bigg|\Phi_{3}(u_{1},v_{1},w_{1})-\Phi_{3}(u_{2},v_{2},w_{2})\bigg|\\
&=\bigg|\int_{0}^{t}(t-\theta)^{\alpha-1}\mathcal{P}_{\alpha,\beta}(t-\theta)\bigg(u_{1}\cdot\nabla w_{1} + \big((-\Delta)^{\frac{2-\beta}{2}}w_{1}\big)v_{1}\bigg) \, d\theta\\
&\quad-\int_{0}^{t}(t-\theta)^{\alpha-1}\mathcal{P}_{\alpha,\beta}(t-\theta)\bigg(u_{2}\cdot\nabla w_{2} + \big((-\Delta)^{\frac{2-\beta}{2}}w_{2}\big)v_{2}\bigg) \, d\theta\bigg|\\
&=\bigg|\int_{0}^{t}(t-\theta)^{\alpha-1}\mathcal{P}_{\alpha,\beta}(t-\theta)
\bigg(\tilde{u}\nabla w_{1}+u_{2}\nabla \tilde{w}+\big((-\Delta)^{\frac{2-\beta}{2}}\tilde{w}\big)v_{1}+\big((-\Delta)^{\frac{2-\beta}{2}}w_{2}\big)\tilde{v}\bigg)\, d\theta\bigg|.
\end{align*}
Therefore, from Lemma \ref{selfmap}, we get
\begin{align*}
&\bigg\|\Phi_{1}(u_{1},v_{1},w_{1})-\Phi_{1}(u_{2},v_{2},w_{2})\bigg\|_{\mathbb{Y}_{1}}\\
&\lesssim\big\|\tilde{u}\big\|_{\mathbb{Y}_{1}}\bigg(\big\|u_{1}\big\|_{\mathbb{Y}_{1}}
+\big\|u_{2}\big\|_{\mathbb{Y}_{2}}\bigg)+\big\|\tilde{v}\big\|_{\mathbb{Y}_{2}}\big\|\nabla\phi\big\|_{d-\lambda,\lambda}\\
&\lesssim\kappa\bigg(\big\|\tilde{u}\big\|_{\mathbb{Y}_{1}}+\big\|\tilde{v}\big\|_{\mathbb{Y}_{2}}\bigg)
\end{align*}
and
\begin{align*}
&\bigg\|\Phi_{2}(u_{1},v_{1},w_{1})-\Phi_{2}(u_{2},v_{2},w_{2})\bigg\|_{\mathbb{Y}_{2}}\\
&\lesssim\bigg(\big\|\tilde{u}\big\|_{\mathbb{Y}_{1}}\big\|v_{1}\big\|_{\mathbb{Y}_{2}}+
\big\|u_{2}\big\|_{\mathbb{Y}_{1}}\big\|\tilde{v}\big\|_{\mathbb{Y}_{2}}+\big\|\tilde{v}\big\|_{\mathbb{Y}_{2}}
\big\|w_{1}\big\|_{\mathbb{Y}_{3}}+\big\|v_{2}\big\|_{\mathbb{Y}_{2}}\big\|\tilde{w}\big\|_{\mathbb{Y}_{3}}
\bigg)\\
&\lesssim\kappa\bigg(\big\|\tilde{u}\big\|_{\mathbb{Y}_{1}}+\big\|\tilde{v}\big\|_{\mathbb{Y}_{2}}
+\big\|\tilde{w}\big\|_{\mathbb{Y}_{3}}\bigg),
\end{align*}
\begin{align*}
&\bigg\|\Phi_{3}(u_{1},v_{1},w_{1})-\Phi_{3}(u_{2},v_{2},w_{2})\bigg\|_{\mathbb{Y}_{2}}\\
&\lesssim\bigg(\big\|\tilde{u}\big\|_{\mathbb{Y}_{1}}\big\|w_{1}\big\|_{\mathbb{Y}_{3}}+
\big\|u_{2}\big\|_{\mathbb{Y}_{1}}\big\|\tilde{w}\big\|_{\mathbb{Y}_{3}}+\big\|\tilde{w}\big\|_{\mathbb{Y}_{3}}
\big\|v_{2}\big\|_{\mathbb{Y}_{2}}+\big\|w_{2}\big\|_{\mathbb{Y}_{3}}\big\|\tilde{v}\big\|_{\mathbb{Y}_{2}}
\bigg)\\
&\lesssim\kappa\bigg(\big\|\tilde{u}\big\|_{\mathbb{Y}_{1}}+\big\|\tilde{v}\big\|_{\mathbb{Y}_{2}}
+\big\|\tilde{w}\big\|_{\mathbb{Y}_{3}}\bigg).
\end{align*}
Therefore, note that the constant $\kappa$ is small enough, we get that
\begin{align*}
&\bigg\|\Phi\big(u_{1},v_{1}.w_{1}\big)-\Phi\big(u_{2},v_{2}.w_{2}\big)\bigg\|_{\mathbb{F}}\\
&\lesssim\bigg\|\Phi_{1}\big(u_{1},v_{1}.w_{1}\big)-\Phi_{1}\big(u_{2},v_{2}.w_{2}\big)\bigg\|_{\mathbb{Y}_{1}}
+\bigg\|\Phi_{2}\big(u_{1},v_{1}.w_{1}\big)-\Phi_{2}\big(u_{2},v_{2}.w_{2}\big)\bigg\|_{\mathbb{Y}_{2}}\\
&\quad +\bigg\|\Phi_{3}\big(u_{1},v_{1}.w_{1}\big)-\Phi_{3}\big(u_{2},v_{2}.w_{2}\big)\bigg\|_{\mathbb{Y}_{3}}\\
&\lesssim C\kappa\bigg(\big\|\tilde{u}\big\|_{\mathbb{Y}_{1}}+\big\|\tilde{v}\big\|_{\mathbb{Y}_{2}}
+\big\|\tilde{w}\big\|_{\mathbb{Y}_{3}}\bigg)<\frac{1}{2}
\bigg(\big\|\tilde{u}\big\|_{\mathbb{Y}_{1}}+\big\|\tilde{v}\big\|_{\mathbb{Y}_{2}}
+\big\|\tilde{w}\big\|_{\mathbb{Y}_{3}}\bigg),
\end{align*}
thus the operator $\Phi$ is a contraction operator. The proof is thus completed.
\end{proof}
\begin{theorem}\label{global mild solution exist}
Assuming conditions \eqref{canshufanwei1} and \eqref{canshufanwei2} are satisfied, and there exists a small positive constant $\kappa>0$ such that
\[
\big\|(u_{0},v_{0},\nabla w_{0})\big\|_{\dot{\mathcal{N}}^{-\beta_{1}}_{r_{1},\lambda,\infty}
\times\dot{\mathcal{N}}^{-\beta_{2}}_{r_{2},\lambda,\infty}\times\dot{\mathcal{N}}^{-\beta_{3}}_{r_{3},\lambda,\infty}}
+\big\|w_{0}\big\|_{\mathcal{M}^{2-\beta}_{\frac{d-\lambda}{2-\beta},\lambda}}+\big\|\nabla\phi\big\|_{d-\lambda,\lambda}\leq\kappa,
\]
the system \eqref{Eq:K-S-N-S} exist unique global mild solution $\big(u,v,w\big)\in \mathbb{F}_{\kappa}$.
\end{theorem}
\begin{proof}
Combine Lemma \ref{selfmap}, Lemma \ref{contraction operator}, the Theorem \ref{global mild solution exist} is obvious. The proof is thus completed.
\end{proof}

\begin{remark}
The Besov-Morrey space encompasses measures that are concentrated on smooth manifolds of $\lambda$-dimension in $\mathbb{R}^{d}$. As a case in point, the initial data $n_{0}$ may belong to $\dot{\mathcal{N}}^{-\beta_{2}}_{r_{2},\lambda,\infty}$, representing such measures. Specifically, for $d=2$, it is permissible to select $\lambda=0$, with $n_{0}$ being a measure localized on a countable set of points-equivalently, a countable aggregation of Dirac deltas. In the scenario where $d=3$, it becomes relevant to consider measures concentrated on structures like filaments and rings, as well as more generally on smooth curves, as possible candidates for the initial data $n_{0}$. (refer to \rm{\cite{L. Ferreira}} for examples).
\end{remark}
\section{asymptotic behavior}
In this section, we consider the asymptotic behavior of solutions derived by the system \eqref{Eq:K-S-N-S}. We can construct the following Theorem.
\begin{theorem}
Let $(u_{1},v_{1},w_{1})$ and $(u_{2},v_{2},w_{2})$ be two global mild solutions of the system \eqref{Eq:K-S-N-S} determined by {\rm Theorem \ref{global mild solution exist}}, with initial values $(u_{1},v_{1},w_{1})|_{t=0}=(u_{10},v_{10},w_{10})$ and $(u_{2},v_{2},w_{2})|_{t=0}=(u_{20},v_{20},w_{20})$, respectively. We define
\begin{align*}
f(t):&=t^{\frac{\chi_{1}}{2}}\big\|\mathcal{S}_{\alpha,\beta}(t)(u_{10}-u_{20})\big\|_{q_{1},\lambda}
+t^{\frac{\chi_{2}}{2}}\big\|\mathcal{S}_{\alpha,\beta}(t)(v_{10}-v_{20})\big\|_{q_{2},\lambda}\\
&\quad +t^{\frac{\chi_{3}}{2}}\big\|\mathcal{S}_{\alpha,\beta}(t)(\nabla w_{10}-\nabla w_{20})\big\|_{q_{3},\lambda}+\big\|\mathcal{S}_{\alpha,\beta}(t)(w_{10}- w_{20})\big\|_{\mathcal{M}^{2-\beta}_{\frac{d-\lambda}{2-\beta},\lambda}}\\
&\quad +\big\|\mathcal{S}_{\alpha,\beta}(t)\big(u_{10}-u_{20},v_{10}-v_{20},\nabla w_{10}-\nabla w_{20}\big)\big\|_{\dot{\mathcal{N}}^{-\beta_{1}}_{r_{1},\lambda,\infty}\times\dot{\mathcal{N}}^{-\beta_{2}}_{r_{2},\lambda,\infty}
\times\dot{\mathcal{N}}^{-\beta_{3}}_{r_{3},\lambda,\infty}},
\end{align*}
and
\begin{align*}
g(t):&=t^{\frac{\chi_{1}}{2}}\big\|u_{1}-u_{2}\big\|_{\mathcal{M}_{q_{1},\lambda}(\mathbb{R}^{d})}
+t^{\frac{\chi_{2}}{2}}\big\|(v_{1}-v_{2})\big\|_{\mathcal{M}_{q_{2},\lambda}(\mathbb{R}^{d})}\\
&\quad+t^{\frac{\chi_{3}}{2}}\big\|\nabla w_{1}-\nabla w_{2}\big\|_{\mathcal{M}_{q_{3},\lambda}(\mathbb{R}^{d})}+\big\|w_{1}- w_{2}\big\|_{\mathcal{M}^{2-\beta}_{\frac{d-\lambda}{2-\beta},\lambda}(\mathbb{R}^{d})}\\
&\quad +\big\|\big(u_{1}-u_{2},v_{1}-v_{2},\nabla w_{1}-\nabla w_{2}\big)\big\|_{\dot{\mathcal{N}}^{-\beta_{1}}_{r_{1},\lambda,\infty}\times\dot{\mathcal{N}}^{-\beta_{2}}_{r_{2},\lambda,\infty}
\times\dot{\mathcal{N}}^{-\beta_{3}}_{r_{3},\lambda,\infty}}.
\end{align*}
Then, it holds that
$$
\lim_{t\rightarrow\infty}f(t)=0 \text{ if and only if } \lim_{t\rightarrow\infty}g(t)=0.
$$
\end{theorem}
\begin{proof}
For Theorem \ref{global mild solution exist}, we have $\left\|\big(u_{1},v_{1},w_{1}\big)\right\|_{\mathbb{F}}\lesssim \kappa$ and $\left\|\big(u_{2},v_{2},w_{2}\big)\right\|_{\mathbb{F}}\lesssim \kappa$. We denote that
$\big(\tilde{u},\tilde{v},\tilde{w}\big)=\big(u_{1}-u_{2},v_{1}-v_{2},w_{1}-w_{2}\big)$, thus we have
\begin{align*}
&u_{1}-u_{2}\\&=\mathcal{S}_{\alpha,\beta}(t)(u_{10}-u_{20})\\
&\quad+\int_{0}^{t}(t-\theta)^{\alpha-1}\mathcal{P}_{\alpha,\beta}(t-\theta)\mathbb{P}
\left(\tilde{u}\cdot\nabla u_{1}+u_{2}\nabla\tilde{u} + \tilde{v}\nabla\phi\right) \, d\theta,\\
&v_{1}-v_{2}\\&=\mathcal{S}_{\alpha,\beta}(t)(v_{10}-v_{20})\\
&\quad+\int_{0}^{t}(t-\theta)^{\alpha-1}\mathcal{P}_{\alpha,\beta}(t-\theta)\left(\tilde{u}\cdot\nabla v_{1}+u_{2}\cdot\nabla \tilde{v} + \nabla\cdot\big(\tilde{v}\nabla(B(w_{1}))+v_{2}\nabla(B(\tilde{w})\big)\right) \, d\theta,\\
&w_{1}-w_{2}\\&=\mathcal{S}_{\alpha,\beta}(t)(w_{10}-w_{20})\\
&\quad+\int_{0}^{t}(t-\theta)^{\alpha-1}\mathcal{P}_{\alpha,\beta}(t-\theta)
\left(\tilde{u}\nabla w_{1}+u_{2}\nabla \tilde{w}+\big((-\Delta)^{\frac{2-\beta}{2}}\tilde{w}\big)v_{1}+\big((-\Delta)^{\frac{2-\beta}{2}}w_{2}\big)\tilde{v}\right)\, d\theta.
\end{align*}
If $\lim_{t\rightarrow\infty}f(t)=0$, we  will divide the following steps to prove $\lim_{t\rightarrow\infty}g(t)=0$.

\textbf{Step 1.} We estimate the norm
$$
\big\|\big(u_{1}-u_{2},v_{1}-v_{2},\nabla w_{1}-\nabla w_{2}\big)\big\|_{\dot{\mathcal{N}}^{-\beta_{1}}_{r_{1},\lambda,\infty}\times\dot{\mathcal{N}}^{-\beta_{2}}_{r_{2},\lambda,\infty}
\times\dot{\mathcal{N}}^{-\beta_{3}}_{r_{3},\lambda,\infty}}.
$$
Note that
\begin{align*}
&\big\|\big(u_{1}-u_{2},v_{1}-v_{2},\nabla w_{1}-\nabla w_{2}\big)\big\|_{\dot{\mathcal{N}}^{-\beta_{1}}_{r_{1},\lambda,\infty}\times\dot{\mathcal{N}}^{-\beta_{2}}_{r_{2},\lambda,\infty}
\times\dot{\mathcal{N}}^{-\beta_{3}}_{r_{3},\lambda,\infty}}\\
&\lesssim\big\|\mathcal{S}_{\alpha,\beta}(t)\big(u_{10}-u_{20},n_{10}-n_{20},\nabla c_{10}-\nabla c_{20}\big)\big\|_{\dot{\mathcal{N}}^{-\beta_{1}}_{r_{1},\lambda,\infty}\times\dot{\mathcal{N}}^{-\beta_{2}}_{r_{2},\lambda,\infty}
\times\dot{\mathcal{N}}^{-\beta_{3}}_{r_{3},\lambda,\infty}}\\
&\quad+ J_{1}+J_{2}+J_{3},
\end{align*}
where
\begin{align*}
J_{1}&=\left\|\int_{0}^{t}(t-\theta)^{\alpha-1}\mathcal{P}_{\alpha,\beta}(t-\theta)\mathbb{P}
\left(\tilde{u}\cdot\nabla u_{1}+u_{2}\nabla\tilde{u} + \tilde{v}\nabla\phi\right) \, d\theta\right\|_{\dot{\mathcal{N}}^{-\beta_{1}}_{r_{1},\lambda,\infty}},\\
J_{2}&=\bigg\|\int_{0}^{t}(t-\theta)^{\alpha-1}\mathcal{P}_{\alpha,\beta}(t-\theta)\\
&\qquad\times\left(\tilde{u}\cdot\nabla v_{1}+u_{2}\cdot\nabla \tilde{v} + \nabla\cdot\big(\tilde{v}\nabla(B(w_{1}))+v_{2}\nabla(B(\tilde{w})\big)\right) \, d\theta\bigg\|_{\dot{\mathcal{N}}^{-\beta_{2}}_{r_{2},\lambda,\infty}},\\
J_{3}&=\bigg\|\int_{0}^{t}(t-\theta)^{\alpha-1}\mathcal{P}_{\alpha,\beta}(t-\theta)\\
&\qquad\times\left(\tilde{u}\nabla w_{1}+u_{2}\nabla \tilde{w}+\big((-\Delta)^{\frac{2-\beta}{2}}\tilde{w}\big)v_{1}+\big((-\Delta)^{\frac{2-\beta}{2}}w_{2}\big)\tilde{v}\right)\, d\theta\bigg\|_{\dot{\mathcal{N}}^{-\beta_{3}}_{r_{3},\lambda,\infty}}.
\end{align*}
For $J_{1}$ and any $0<\varrho<1$, combine above Lemma \ref{fractional heat kenel lemma}, Lemma \ref{selfmap}, we can get
\begin{align*}
J_{1}&=\left\|\int_{0}^{t}(t-\theta)^{\alpha-1}\mathcal{P}_{\alpha,\beta}(t-\theta)\mathbb{P}
\left(\tilde{u}\cdot\nabla u_{1}+u_{2}\nabla\tilde{u} + \tilde{v}\nabla\phi\right) \, d\theta\right\|_{\dot{\mathcal{N}}^{-\beta_{1}}_{r_{1},\lambda,\infty}}\\
&\lesssim\left\|\int_{0}^{t}(t-\theta)^{\alpha-1}\mathcal{P}_{\alpha,\beta}(t-\theta)\mathbb{P}
\left(\tilde{u}\cdot\nabla u_{1}+u_{2}\nabla\tilde{u} + \tilde{v}\nabla\phi\right) \, d\theta\right\|_{\mathcal{M}_{\frac{d-\lambda}{\beta-1},\lambda}}\\
&\lesssim\left(\int_{0}^{\varrho t}+\int_{\varrho t}^{t}\right)(t-\theta)^{2\alpha-\frac{2\alpha}{\beta}-\frac{2\alpha(d-\lambda)}{\beta q_{1}}-1}\big\|\tilde{u}(\theta)\big\|_{q_{1},\lambda}\left(\big\|u_{1}(\theta)\big\|_{q_{1},\lambda}
+\big\|u_{2}(\theta)\big\|_{q_{1},\lambda}\right)d\theta\\
&\quad+\left(\int_{0}^{\varrho t}+\int_{\varrho t}^{t}\right)(t-\theta)^{2\alpha-\frac{2\alpha}{\beta}-\frac{2\alpha(d-\lambda)}{\beta q_{2}}-1}\big\|\tilde{v}(\theta)\big\|_{q_{2},\lambda}\big\|\nabla\phi\big\|_{d-\lambda,\lambda}d\theta\\
&\lesssim\left(\big\|u_{1}\big\|_{\mathbb{X}_{1}}+\big\|u_{2}\big\|_{\mathbb{X}_{1}}\right)\int_{0}^{\varrho t}(t-\theta)^{2\alpha-\frac{2\alpha}{\beta}-\frac{2\alpha(d-\lambda)}{\beta q_{1}}-1}\theta^{-\chi_{1}}\left(\theta^{\frac{\chi_{1}}{2}}\big\|\tilde{u}(\theta)\big\|_{q_{1},\lambda}\right)d\theta\\
&\quad +\big\|\nabla\phi\big\|_{d-\lambda,\lambda}\int_{0}^{\varrho t}(t-\theta)^{2\alpha-\frac{2\alpha}{\beta}-\frac{2\alpha(d-\lambda)}{\beta q_{2}}-1}\theta^{-\frac{\chi_{2}}{2}}\left(\theta^{\frac{\chi_{2}}{2}}\big\|\tilde{v}(\theta)\big\|_{q_{2},\lambda}\right)d\theta\\
&\quad +\left(\big\|u_{1}\big\|_{\mathbb{X}_{1}}+\big\|u_{2}\big\|_{\mathbb{X}_{1}}\right)\left(\sup_{\varrho t\leq \theta\leq t}
\big\|\tilde{u}(\theta)\big\|_{q_{1},\lambda}\right)+\big\|\nabla\phi\big\|_{d-\lambda,\lambda}\left(\sup_{\varrho t\leq \theta\leq t}
\big\|\tilde{v}(\theta)\big\|_{q_{2},\lambda}\right)\\
&\lesssim \kappa\int_{0}^{\varrho t}(t-\theta)^{2\alpha-\frac{2\alpha}{\beta}-\frac{2\alpha(d-\lambda)}{\beta q_{1}}-1}\theta^{-\chi_{1}}\left(\theta^{\frac{\chi_{1}}{2}}\big\|\tilde{u}(\theta)\big\|_{q_{1},\lambda}\right)d\theta\\
&\quad +\kappa\int_{0}^{\varrho t}(t-\theta)^{2\alpha-\frac{2\alpha}{\beta}-\frac{2\alpha(d-\lambda)}{\beta q_{2}}-1}\theta^{-\frac{\chi_{2}}{2}}\left(\theta^{\frac{\chi_{2}}{2}}\big\|\tilde{v}(\theta)\big\|_{q_{2},\lambda}\right)d\theta\\
&\quad +\kappa
\left(\left(\sup_{\varrho t\leq \theta\leq t}\theta^{\frac{\chi_{1}}{2}}
\big\|\tilde{u}(\theta)\big\|_{q_{1},\lambda}\right)+\left(\sup_{\varrho t\leq \theta\leq t}
\theta^{\frac{\chi_{2}}{2}}\big\|\tilde{v}(\theta)\big\|_{q_{2},\lambda}\right)\right),
\end{align*}
similar, we can get that
\begin{align*}
J_{2}&=\bigg\|\int_{0}^{t}(t-\theta)^{\alpha-1}\mathcal{P}_{\alpha,\beta}(t-\theta)\left(\tilde{u}\cdot\nabla v_{1}+u_{2}\cdot\nabla \tilde{v} + \nabla\cdot\big(\tilde{v}\nabla(B(w_{1}))+v_{2}\nabla(B(\tilde{w})\big)\right) \, d\theta\bigg\|_{\dot{\mathcal{N}}^{-\beta_{2}}_{r_{2},\lambda,\infty}}\\
&\lesssim\bigg\|\int_{0}^{t}(t-\theta)^{\alpha-1}\mathcal{P}_{\alpha,\beta}(t-\theta)\left(\tilde{u}\cdot\nabla v_{1}+u_{2}\cdot\nabla \tilde{v} + \nabla\cdot\big(\tilde{v}\nabla(B(w_{1}))+v_{2}\nabla(B(\tilde{w})\big)\right) \, d\theta\bigg\|_{\mathcal{M}_{\frac{d-\lambda}{2\beta-2},\lambda}}\\
&\lesssim\kappa\int_{0}^{\varrho t}(t-\theta)^{3\alpha-\frac{3\alpha}{\beta}-\frac{\alpha(d-\lambda)}{\beta q_{1}}-\frac{\alpha(d-\lambda)}{\beta q_{2}}-1}\theta^{-\frac{\chi_{1}+\chi_{2}}{2}}\left(\theta^{\frac{\chi_{1}}{2}}\big\|\tilde{u}(\theta)\big\|_{q_{1},\lambda}+
\theta^{\frac{\chi_{2}}{2}}\big\|\tilde{v}(\theta)\big\|_{q_{2},\lambda}\right)d\theta\\
&\quad +\kappa\int_{0}^{\varrho t}(t-\theta)^{2\alpha-\frac{\alpha}{\beta}-\frac{\alpha(d-\lambda)}{\beta q_{2}}-\frac{\alpha(d-\lambda)}{\beta q_{3}}-1}\theta^{-\frac{\chi_{2}+\chi_{3}}{2}}\left(\theta^{\frac{\chi_{2}}{2}}\big\|\tilde{v}(\theta)\big\|_{q_{2},\lambda}+
\theta^{\frac{\chi_{3}}{2}}\big\|\nabla\tilde{w}(\theta)\big\|_{q_{3},\lambda}\right)d\theta\\
&\quad +\kappa\left(\sup_{\varrho t\leq \theta\leq t}\left(\theta^{\frac{\chi_{1}}{2}}\big\|\tilde{u}(\theta)\big\|_{q_{1},\lambda}\right)
+\sup_{\varrho t\leq \theta\leq t}\left(\theta^{\frac{\chi_{2}}{2}}\big\|\tilde{v}(\theta)\big\|_{q_{2},\lambda}\right)
+\sup_{\varrho t\leq \theta\leq t}\left(\theta^{\frac{\chi_{3}}{2}}\big\|\nabla\tilde{w}(\theta)\big\|_{q_{3},\lambda}\right)\right),
\end{align*}
and
\begin{align*}
J_{3}&=\bigg\|\int_{0}^{t}(t-\theta)^{\alpha-1}\nabla \mathcal{P}_{\alpha,\beta}(t-\theta)\left(\tilde{u}\nabla w_{1}+u_{2}\nabla \tilde{w}+\big((-\Delta)^{\frac{2-\beta}{2}}\tilde{w}\big)v_{1}+\big((-\Delta)^{\frac{2-\beta}{2}}w_{2}\big)\tilde{v}\right)\, d\theta\bigg\|_{\dot{\mathcal{N}}^{-\beta_{3}}_{r_{3},\lambda,\infty}}\\
&\lesssim\bigg\|\int_{0}^{t}(t-\theta)^{\alpha-1}\nabla \mathcal{P}_{\alpha,\beta}(t-\theta)\left(\tilde{u}\nabla w_{1}+u_{2}\nabla \tilde{w}+\big((-\Delta)^{\frac{2-\beta}{2}}\tilde{w}\big)v_{1}+\big((-\Delta)^{\frac{2-\beta}{2}}w_{2}\big)\tilde{v}\right)\, d\theta\bigg\|_{\mathcal{M}_{d-\lambda,\lambda}}\\
&\lesssim\kappa\int_{0}^{\varrho t}(t-\theta)^{\alpha-\frac{\alpha(d-\lambda)}{\beta q_{1}}-\frac{\alpha(d-\lambda)}{\beta q_{3}}-1}
\theta^{-\frac{\chi_{1}+\chi_{3}}{2}}\left(\theta^{\frac{\chi_{1}}{2}}\big\|\tilde{u}(\theta)\big\|_{q_{1},\lambda}
+\theta^{\frac{\chi_{3}}{2}}\big\|\nabla\tilde{w}(\theta)\big\|_{q_{3},\lambda}\right)d\theta\\
&\quad +\kappa\int_{0}^{\varrho t}(t-\theta)^{2\alpha-\frac{2\alpha}{\beta}-\frac{2\alpha(d-\lambda)}{\beta q_{2}}-1}
\theta^{-\frac{\chi_{2}}{2}}\left(\theta^{\frac{\chi_{2}}{2}}\big\|\tilde{v}(\theta)\big\|_{q_{2},\lambda}
+\big\|\tilde{w}(\theta)\big\|_{\mathcal{M}^{2-\beta}_{\frac{d-\lambda}{2-\beta},\lambda}}\right)d\theta\\
&\quad +\kappa\bigg(\sup_{\varrho t\leq \theta\leq t}\left(\theta^{\frac{\chi_{1}}{2}}\big\|\tilde{u}(\theta)\big\|_{\mathcal{M}_{q_{1},\lambda}}\right)
+\sup_{\varrho t\leq \theta\leq t}\left(\theta^{\frac{\chi_{3}}{2}}\big\|\nabla\tilde{w}(\theta)\big\|_{\mathcal{M}_{q_{3},\lambda}}\right)+\sup_{\varrho t\leq \theta\leq t}\left(\theta^{\frac{\chi_{2}}{2}}\big\|\tilde{v}(\theta)\big\|_{\mathcal{M}_{q_{2},\lambda}}\right)
\\
&\quad+\sup_{\varrho t\leq \theta\leq t}\big(\big\|\tilde{w}(\theta)\big\|_{\mathcal{M}^{2-\beta}_{\frac{d-\lambda}{2-\beta},\lambda}}\big)\bigg).
\end{align*}
Thus, we obtain that
\begin{align*}
&\big\|\big(u_{1}-u_{2},v_{1}-v_{2},\nabla w_{1}-\nabla w_{2}\big)\big\|_{\dot{\mathcal{N}}^{-\beta_{1}}_{r_{1},\lambda,\infty}\times\dot{\mathcal{N}}^{-\beta_{2}}_{r_{2},\lambda,\infty}
\times\dot{\mathcal{N}}^{-\beta_{3}}_{r_{3},\lambda,\infty}}\\
&\lesssim\big\|\mathcal{S}_{\alpha,\beta}(t)\big(u_{10}-u_{20},v_{10}-v_{20},\nabla w_{10}-\nabla w_{20}\big)\big\|_{\dot{\mathcal{N}}^{-\beta_{1}}_{r_{1},\lambda,\infty}\times\dot{\mathcal{N}}^{-\beta_{2}}_{r_{2},\lambda,\infty}
\times\dot{\mathcal{N}}^{-\beta_{3}}_{r_{3},\lambda,\infty}}\\
&\quad +\kappa\int_{0}^{\varrho}(1-\theta)^{2\alpha-\frac{2\alpha}{\beta}-\frac{2\alpha(d-\lambda)}{\beta q_{1}}-1}\theta^{-\chi_{1}}\left((t\theta)^{\frac{\chi_{1}}{2}}\big\|\tilde{u}(t\theta)\big\|_{\mathcal{M}_{q_{1},\lambda}}\right)d\theta\\
&\quad +\kappa\int_{0}^{\varrho}(1-\theta)^{2\alpha-\frac{2\alpha}{\beta}-\frac{2\alpha(d-\lambda)}{\beta q_{2}}-1}\theta^{-\frac{\chi_{2}}{2}}\left((t\theta)^{\frac{\chi_{2}}{2}}\big\|\tilde{v}(t\theta)\big\|_{\mathcal{M}_{q_{2},\lambda}}\right)d\theta\\
&\quad +\kappa\int_{0}^{\varrho}(1-\theta)^{3\alpha-\frac{3\alpha}{\beta}-\frac{\alpha(d-\lambda)}{\beta q_{1}}-\frac{\alpha(d-\lambda)}{\beta q_{2}}-1}\theta^{-\frac{\chi_{1}+\chi_{2}}{2}}\left((t\theta)^{\frac{\chi_{1}}{2}}\big\|\tilde{u}(t\theta)\big\|_{\mathcal{M}_{q_{1},\lambda}}+
(t\theta)^{\frac{\chi_{2}}{2}}\big\|\tilde{v}(t\theta)\big\|_{\mathcal{M}_{q_{2},\lambda}}\right)d\theta\\
&\quad +\kappa\int_{0}^{\varrho}(1-\theta)^{2\alpha-\frac{\alpha}{\beta}-\frac{\alpha(d-\lambda)}{\beta q_{2}}-\frac{\alpha(d-\lambda)}{\beta q_{3}}-1}\theta^{-\frac{\chi_{2}+\chi_{3}}{2}}\left((t\theta)^{\frac{\chi_{3}}{2}}\big\|\nabla\tilde{w}(t\theta)\big\|_{\mathcal{M}_{q_{3},\lambda}}+
(t\theta)^{\frac{\chi_{2}}{2}}\big\|\tilde{v}(t\theta)\big\|_{\mathcal{M}_{q_{2},\lambda}}\right)d\theta\\
&\quad+\kappa\int_{0}^{\varrho}(1-\theta)^{\alpha-\frac{\alpha(d-\lambda)}{\beta q_{1}}-\frac{\alpha(d-\lambda)}{\beta q_{3}}-1}
s^{-\frac{\chi_{1}+\chi_{3}}{2}}\left((ts)^{\frac{\chi_{3}}{2}}\big\|\nabla\tilde{c}(t\theta)\big\|_{\mathcal{M}_{q_{3},\lambda}}
+(t\theta)^{\frac{\chi_{1}}{2}}\big\|\tilde{u}(t\theta)\big\|_{\mathcal{M}_{q_{1},\lambda}}\right)d\theta\\
&\quad +\kappa\int_{0}^{\varrho}(1-\theta)^{2\alpha-\frac{2\alpha}{\beta}-\frac{2\alpha(d-\lambda)}{\beta q_{2}}-1}
\theta^{-\frac{\chi_{2}}{2}}\left((t\theta)^{\frac{\chi_{2}}{2}}\big\|\tilde{v}(t\theta)\big\|_{\mathcal{M}_{q_{2},\lambda}}
+\big\|\tilde{w}(t\theta)\big\|_{\mathcal{M}^{2-\beta}_{\frac{d-\lambda}{2-\beta},\lambda}}\right)d\theta\\
&\quad +\kappa\bigg(\sup_{\varrho t\leq \theta\leq t}\left(\theta^{\frac{\chi_{1}}{2}}\big\|\tilde{u}(\theta)\big\|_{\mathcal{M}_{q_{1},\lambda}(\mathbb{R}^{d})}\right)
+\sup_{\varrho t\leq \theta\leq t}\left(\theta^{\frac{\chi_{3}}{2}}\big\|\nabla\tilde{w}(\theta)\big\|_{\mathcal{M}_{q_{3},\lambda}(\mathbb{R}^{d})}\right)+\sup_{\varrho t\leq \theta\leq t}\left(\theta^{\frac{\chi_{2}}{2}}\big\|\tilde{v}(\theta)\big\|_{\mathcal{M}_{q_{2},\lambda}(\mathbb{R}^{d})}\right)\\
&\quad+\sup_{\varrho t\leq \theta\leq t}\big(\big\|\tilde{w}(\theta)\big\|_{\mathcal{M}^{2-\beta}_{\frac{d-\lambda}{2-\beta},\lambda}}\big)\bigg).
\end{align*}

\textbf{Step 2.} We estimate the norm
$$
\left\|\left(t^{\frac{\chi_{1}}{2}}\left(u_{1}-u_{2}\right),t^{\frac{\chi_{2}}{2}}\left(v_{1}-v_{2}\right),
t^{\frac{\chi_{3}}{2}}\left(\nabla w_{1}-\nabla w_{2}\right),w_{1}-w_{2}\right)\right\|_{\mathcal{M}_{q_{1},\lambda}\times\mathcal{M}_{q_{2},\lambda}\times\mathcal{M}_{q_{3},\lambda}
\times\mathcal{M}^{2-\beta}_{\frac{d-\lambda}{2-\beta},\lambda}}.
$$
Note that
\begin{align*}
&\big\|\big(u_{1}-u_{2},v_{1}-v_{2},\nabla w_{1}-\nabla w_{2},w_{1}-w_{2}\big)\big\|_{\mathcal{M}_{q_{1},\lambda}\times\mathcal{M}_{q_{2},\lambda}\times\mathcal{M}_{q_{3},\lambda}
\times\mathcal{M}^{2-\beta}_{\frac{d-\lambda}{2-\beta},\lambda}}\\
&\lesssim\big\|\mathcal{S}_{\alpha,\beta}(t)\big(u_{10}-u_{20},v_{10}-v_{20},\nabla w_{10}-\nabla w_{20},w_{10}-w_{20}\big)\big\|_{\mathcal{M}_{q_{1},\lambda}\times\mathcal{M}_{q_{2},\lambda}\times\mathcal{M}_{q_{3},\lambda}
\times\mathcal{M}^{2-\beta}_{\frac{d-\lambda}{2-\beta},\lambda}}\\
&\quad+ I_{1}+I_{2}+I_{3}+I_{4},
\end{align*}
where
\begin{align*}
I_{1}&=\left\|\int_{0}^{t}(t-\theta)^{\alpha-1}\mathcal{P}_{\alpha,\beta}(t-\theta)\mathbb{P}
\left(\tilde{u}\cdot\nabla u_{1}+u_{2}\nabla\tilde{u} + \tilde{v}\nabla\phi\right) \, d\theta\right\|_{\mathcal{M}_{q_{1},\lambda}},\\
I_{2}&=\bigg\|\int_{0}^{t}(t-\theta)^{\alpha-1}\mathcal{P}_{\alpha,\beta}(t-\theta)\left(\tilde{u}\cdot\nabla v_{1}+u_{2}\cdot\nabla \tilde{v} + \nabla\cdot\big(\tilde{v}\nabla(B(w_{1}))+v_{2}\nabla(B(\tilde{w})\big)\right) \, d\theta\bigg\|_{\mathcal{M}_{q_{2},\lambda}},\\
I_{3}&=\bigg\|\int_{0}^{t}(t-\theta)^{\alpha-1}\nabla \mathcal{P}_{\alpha,\beta}(t-\theta)\left(\tilde{u}\nabla w_{1}+u_{2}\nabla \tilde{w}+\big((-\Delta)^{\frac{2-\beta}{2}}\tilde{w}\big)v_{1}+\big((-\Delta)^{\frac{2-\beta}{2}}w_{2}\big)\tilde{v}\right)\, d\theta\bigg\|_{\mathcal{M}_{q_{3},\lambda}},\\
I_{4}&=\bigg\|\int_{0}^{t}(t-\theta)^{\alpha-1}\mathcal{P}_{\alpha,\beta}(t-\theta)\left(\tilde{u}\nabla w_{1}+u_{2}\nabla \tilde{w}+\big((-\Delta)^{\frac{2-\beta}{2}}\tilde{w}\big)v_{1}+\big((-\Delta)^{\frac{2-\beta}{2}}w_{2}\big)\tilde{v}\right)\, d\theta\bigg\|_{\mathcal{M}^{2-\beta}_{\frac{d-\lambda}{2-\beta},\lambda}}.
\end{align*}
For $I_{1}$ and any $0<\varrho<1$, combine above Lemma \ref{fractional heat kenel lemma}, Lemma \ref{selfmap}, we can get
\begin{align*}
I_{1}&=\left\|\int_{0}^{t}(t-\theta)^{\alpha-1}\mathcal{P}_{\alpha,\beta}(t-\theta)\mathbb{P}
\left(\tilde{u}\cdot\nabla u_{1}+u_{2}\nabla\tilde{u} + \tilde{v}\nabla\phi\right) \, d\theta\right\|_{\mathcal{M}_{q_{1},\lambda}}\\
&\lesssim\left(\int_{0}^{\varrho t}+\int_{\varrho t}^{t}\right)(t-\theta)^{\alpha-\frac{\alpha}{\beta}-\frac{\alpha(d-\lambda)}{\beta}(\frac{2}{q_{1}}-\frac{1}{q_{1}})-1}
\left(\big\|u_{1}(\theta)\otimes\tilde{u}(\theta)\big\|_{\frac{q_{1}}{2},\lambda}
+\big\|u_{2}(\theta)\otimes\tilde{u}(\theta)\big\|_{\frac{q_{1}}{2},\lambda}\right)d\theta\\
&\quad +\left(\int_{0}^{\varrho t}+\int_{\varrho t}^{t}\right)(t-\theta)^{\alpha-\frac{\alpha(d-\lambda)}{\beta}(\frac{1}{q_{2}}+\frac{1}{d-\lambda}-\frac{1}{q_{1}})-1}
\big\|\tilde{v}(\theta)\big\|_{q_{2},\lambda}\big\|\nabla\phi\big\|_{d-\lambda,\lambda}d\theta\\
&\lesssim\kappa\int_{0}^{\varrho t}(t-\theta)^{\alpha-\frac{\alpha}{\beta}-\frac{\alpha(d-\lambda)}{\beta}(\frac{2}{q_{1}}-\frac{1}{q_{1}})-1}\theta^{-\chi_{1}}
\left(\theta^{\frac{\chi_{1}}{2}}\big\|\tilde{u}(\theta)\big\|_{q_{1},\lambda}\right)d\theta\\
&\quad +\kappa\int_{0}^{\varrho t}(t-\theta)^{\alpha-\frac{\alpha(d-\lambda)}{\beta}(\frac{1}{q_{2}}+\frac{1}{d-\lambda}-\frac{1}{q_{1}})-1}
\theta^{-\frac{\chi_{2}}{2}}\left(\theta^{\frac{\chi_{2}}{2}}\big\|\tilde{v}(\theta)\big\|_{q_{2},\lambda}\right)d\theta\\
&\quad +\kappa t^{-\alpha+\frac{\alpha}{\beta}+\frac{\alpha(d-\lambda)}{\beta q_{1}}}\left(\sup_{\varrho t\leq \theta\leq t}\left(\theta^{\frac{\chi_{1}}{2}}\big\|\tilde{u}(\theta)\big\|_{q_{1},\lambda}\right)+\sup_{\varrho t\leq \theta\leq t}\left(\theta^{\frac{\chi_{2}}{2}}\big\|\tilde{v}(\theta)\big\|_{q_{2},\lambda}\right)\right),
\end{align*}
similar, we also have that
\begin{align*}
I_{2}&=\bigg\|\int_{0}^{t}(t-\theta)^{\alpha-1}\mathcal{P}_{\alpha,\beta}(t-\theta)\left(\tilde{u}\cdot\nabla v_{1}+u_{2}\cdot\nabla \tilde{v} + \nabla\cdot\big(\tilde{v}\nabla(B(w_{1}))+v_{2}\nabla(B(\tilde{w})\big)\right) \, d\theta\bigg\|_{\mathcal{M}_{q_{2},\lambda}}\\
&\lesssim\kappa\int_{0}^{\varrho t}(t-\theta)^{\alpha-\frac{\alpha}{\beta}-\frac{\alpha(d-\lambda)}{\beta q_{1}}-1}\theta^{-\frac{\chi_{1}+\chi_{2}}{2}}\left(\theta^{\frac{\chi_{1}}{2}}\big\|\tilde{u}(\theta)\big\|_{q_{1},\lambda}
+\theta^{\frac{\chi_{2}}{2}}\big\|\tilde{v}(\theta)\big\|_{q_{2},\lambda}\right)d\theta\\
&\quad +\kappa\int_{0}^{\varrho t}(t-\theta)^{\frac{\alpha}{\beta}-\frac{\alpha(d-\lambda)}{\beta q_{3}}-1}
\theta^{-\frac{\chi_{2}+\chi_{3}}{2}}\left(\theta^{\frac{\chi_{2}}{2}}\big\|\tilde{v}(\theta)\big\|_{q_{2},\lambda}
+\theta^{\frac{\chi_{3}}{2}}\big\|\nabla\tilde{w}(\theta)\big\|_{q_{3},\lambda}\right)d\theta\\
&\quad +\kappa t^{-2\alpha+\frac{2\alpha}{\beta}+\frac{\alpha(d-\lambda)}{\beta q_{2}}}\bigg(
\sup_{\varrho t\leq \theta\leq t}\left(\theta^{\frac{\chi_{1}}{2}}\big\|\tilde{u}(\theta)\big\|_{q_{1},\lambda}\right)+\sup_{\varrho t\leq \theta\leq t}\left(\theta^{\frac{\chi_{2}}{2}}\big\|\tilde{v}(\theta)\big\|_{q_{2},\lambda}\right)\\
&\quad +\sup_{\varrho t\leq \theta\leq t}\left(\theta^{\frac{\chi_{3}}{2}}\big\|\nabla\tilde{w}(\theta)\big\|_{q_{3},\lambda}\right)\bigg),
\end{align*}
\begin{align*}
I_{3}&=\bigg\|\int_{0}^{t}\nabla(t-\theta)^{\alpha-1}\mathcal{P}_{\alpha,\beta}(t-\theta)\left(\nabla w_{1}\cdot\tilde{u}+\nabla \tilde{w}\cdot u_{2}+\big((-\Delta)^{\frac{2-\beta}{2}}\tilde{w}\big)v_{1}+\big((-\Delta)^{\frac{2-\beta}{2}}w_{2}\big)\tilde{v}\right)\, d\theta\bigg\|_{\mathcal{M}_{q_{3},\lambda}}\\
&\lesssim\kappa\int_{0}^{\varrho t}(t-\theta)^{\alpha-\frac{\alpha}{\beta}-\frac{\alpha(d-\lambda)}{\beta q_{1}}-1}\theta^{-\frac{\chi_{1}+\chi_{3}}{2}}\left(\theta^{\frac{\chi_{1}}{2}}\big\|\tilde{u}(\theta)\big\|_{q_{1},\lambda}
+\theta^{\frac{\chi_{3}}{2}}\big\|\nabla\tilde{w}(\theta)\big\|_{q_{3},\lambda}\right)d\theta\\
&\quad +\kappa\int_{0}^{\varrho t}(t-\theta)^{2\alpha-\frac{3\alpha}{\beta}-\frac{\alpha(d-\lambda)}{\beta q_{2}}+\frac{\alpha(d-\lambda)}{\beta q_{3}}-1}\theta^{-\frac{\chi_{2}}{2}}\left(\theta^{\frac{\chi_{2}}{2}}\big\|\tilde{v}(\theta)\big\|_{q_{2},\lambda}+\big\|\tilde{w}(\theta)
\big\|_{\mathcal{M}^{2-\beta}_{\frac{d-\lambda}{2-\beta},\lambda}}\right)d\theta\\
&\quad +\kappa t^{-\frac{\alpha}{\beta}+\frac{\alpha(d-\lambda)}{\beta q_{3}}}\bigg(
\sup_{\varrho t\leq \theta\leq t}\left(\theta^{\frac{\chi_{1}}{2}}\big\|\tilde{u}(\theta)\big\|_{q_{1},\lambda}\right)+\sup_{\varrho t\leq \theta\leq t}\left(\theta^{\frac{\chi_{2}}{2}}\big\|\tilde{v}(\theta)\big\|_{q_{2},\lambda}\right)\\
&\quad +\sup_{\varrho t\leq \theta\leq t}\left(\theta^{\frac{\chi_{3}}{2}}\big\|\nabla\tilde{w}(\theta)\big\|_{q_{3},\lambda}\right)+\sup_{\varrho t\leq \theta\leq t}\big\|\tilde{w}(\theta)
\big\|_{\mathcal{M}^{2-\beta}_{\frac{d-\lambda}{2-\beta},\lambda}}\bigg),
\end{align*}
and
\begin{align*}
I_{4}&=\bigg\|\int_{0}^{t}(t-\theta)^{\alpha-1}\mathcal{P}_{\alpha,\beta}(t-\theta)\left(\tilde{u}\nabla w_{1}+u_{2}\nabla \tilde{w}+\big((-\Delta)^{\frac{2-\beta}{2}}\tilde{w}\big)v_{1}+\big((-\Delta)^{\frac{2-\beta}{2}}w_{2}\big)\tilde{v}\right)\, d\theta\bigg\|_{\mathcal{M}^{2-\beta}_{\frac{d-\lambda}{2-\beta},\lambda}}\\
&\lesssim\bigg\|(-\Delta)^{\frac{2-\beta}{2}}\int_{0}^{t}(t-\theta)^{\alpha-1}\mathcal{P}_{\alpha,\beta}(t-\theta)\left(\tilde{u}\nabla w_{1}+u_{2}\nabla \tilde{w}+\big((-\Delta)^{\frac{2-\beta}{2}}\tilde{w}\big)v_{1}+\big((-\Delta)^{\frac{2-\beta}{2}}w_{2}\big)\tilde{v}\right)\, d\theta\bigg\|_{\mathcal{M}_{\frac{d-\lambda}{2-\beta},\lambda}}\\
&\lesssim\kappa\int_{0}^{\varrho t}(t-\theta)^{\alpha-\frac{\alpha(d-\lambda)}{\beta q_{1}}-\frac{\alpha(d-\lambda)}{\beta q_{3}}-1}\theta^{-\frac{\chi_{1}+\chi_{3}}{2}}\left(\theta^{\frac{\chi_{1}}{2}}\big\|\tilde{u}(\theta)\big\|_{q_{1},\lambda}
+\theta^{\frac{\chi_{3}}{2}}\big\|\nabla\tilde{w}(\theta)\big\|_{q_{3},\lambda}\right)d\theta\\
&\quad +\kappa\int_{0}^{\varrho t}(t-\theta)^{2\alpha-\frac{2\alpha}{\beta}-\frac{\alpha(d-\lambda)}{\beta q_{2}}-1}\theta^{-\frac{\chi_{2}}{2}}\left(\theta^{\frac{\chi_{2}}{2}}\big\|\tilde{v}(\theta)\big\|_{q_{2},\lambda}+\big\|\tilde{w}(\theta)
\big\|_{\mathcal{M}^{2-\beta}_{\frac{d-\lambda}{2-\beta},\lambda}}\right)d\theta\\
&\quad +\kappa\bigg(
\sup_{\varrho t\leq \theta\leq t}\left(\theta^{\frac{\chi_{1}}{2}}\big\|\tilde{u}(\theta)\big\|_{q_{1},\lambda}\right)+\sup_{\varrho t\leq \theta\leq t}\left(\theta^{\frac{\chi_{2}}{2}}\big\|\tilde{v}(\theta)\big\|_{q_{2},\lambda}\right)\\
&\quad +\sup_{\varrho t\leq \theta\leq t}\left(\theta^{\frac{\chi_{3}}{2}}\big\|\nabla\tilde{w}(\theta)\big\|_{q_{3},\lambda}\right)+\sup_{\varrho t\leq \theta\leq t}\big\|\tilde{w}(\theta)
\big\|_{\mathcal{M}^{2-\beta}_{\frac{d-\lambda}{2-\beta},\lambda}}\bigg).
\end{align*}
Thus, we obtain that
\begin{align*}
&\left\|\left(t^{\frac{\chi_{1}}{2}}\left(u_{1}-u_{2}\right),t^{\frac{\chi_{2}}{2}}\left(v_{1}-v_{2}\right),
t^{\frac{\chi_{3}}{2}}\left(\nabla w_{1}-\nabla w_{2}\right),w_{1}-w_{2}\right)\right\|_{\mathcal{M}_{q_{1},\lambda}\times\mathcal{M}_{q_{2},\lambda}\times\mathcal{M}_{q_{3},\lambda}
\times\mathcal{M}^{2-\beta}_{\frac{d-\lambda}{2-\beta},\lambda}}\\
&\lesssim t^{\frac{\chi_{1}}{2}}\left\|\mathcal{S}_{\alpha,\beta}(t)(u_{10}-u_{20})\right\|_{\mathcal{M}_{q_{1},\lambda}}
+t^{\frac{\chi_{2}}{2}}\left\|\mathcal{S}_{\alpha,\beta}(t)(v_{10}-v_{20})
\right\|_{\mathcal{M}_{q_{2},\lambda}}\\
&\quad +t^{\frac{\chi_{3}}{2}}\left\|\mathcal{S}_{\alpha,\beta}(t)(\nabla w_{10}-\nabla w_{20})\right\|_{\mathcal{M}_{q_{3},\lambda}}
+\left\|\mathcal{S}_{\alpha,\beta}(t)(w_{10}-w_{20})\right\|
_{\mathcal{M}^{2-\beta}_{\frac{d-\lambda}{2-\beta},\lambda}}\\
&\quad +\kappa\bigg(
\sup_{\varrho t\leq \theta\leq t}\left(\theta^{\frac{\chi_{1}}{2}}\big\|\tilde{u}(\theta)\big\|_{\mathcal{M}_{q_{1},\lambda}(\mathbb{R}^{d})}\right)+\sup_{\varrho t\leq \theta\leq t}\left(\theta^{\frac{\chi_{2}}{2}}\big\|\tilde{v}(\theta)\big\|_{\mathcal{M}_{q_{2},\lambda}(\mathbb{R}^{d})}\right)\\
&\quad +\sup_{\varrho t\leq \theta\leq t}\big\|\tilde{w}(\theta)
\big\|_{\mathcal{M}^{2-\beta}_{\frac{d-\lambda}{2-\beta},\lambda}(\mathbb{R}^{d})}+\sup_{\varrho t\leq \theta\leq t}\left(\theta^{\frac{\chi_{3}}{2}}\big\|\nabla\tilde{w}(\theta)\big\|_{\mathcal{M}_{q_{3},\lambda}(\mathbb{R}^{d})}\right)\bigg)
+H_{1}(t)+H_{2}(t),
\end{align*}
where
\begin{align*}
H_{1}(t)&=\kappa\int_{0}^{\varrho}(1-\theta)^{\alpha-\frac{\alpha}{\beta}-\frac{\alpha(d-\lambda)}{\beta}
(\frac{2}{q_{1}}-\frac{1}{q_{1}})-1}\theta^{-\chi_{1}}
\left((t\theta)^{\frac{\chi_{1}}{2}}\big\|\tilde{u}(t\theta)\big\|_{q_{1},\lambda}\right)d\theta\\
&\quad +\kappa\int_{0}^{\varrho}(1-\theta)^{\alpha-\frac{\alpha(d-\lambda)}{\beta}(\frac{1}{q_{2}}+\frac{1}{d-\lambda}-\frac{1}{q_{1}})-1}
\theta^{-\frac{\chi_{2}}{2}}\left((t\theta)^{\frac{\chi_{2}}{2}}\big\|\tilde{v}(t\theta)\big\|_{q_{2},\lambda}\right)d\theta\\
&\quad +\kappa\int_{0}^{\varrho}(1-\theta)^{\alpha-\frac{\alpha}{\beta}-\frac{\alpha(d-\lambda)}{\beta q_{1}}-1}\theta^{-\frac{\chi_{1}+\chi_{2}}{2}}\left((t\theta)^{\frac{\chi_{1}}{2}}
\big\|\tilde{u}(t\theta)\big\|_{\mathcal{M}_{q_{1},\lambda}(\mathbb{R}^{d})}
+(t\theta)^{\frac{\chi_{2}}{2}}\big\|\tilde{v}(t\theta)\big\|_{\mathcal{M}_{q_{2},\lambda}(\mathbb{R}^{d})}\right)d\theta\\
&\quad +\kappa\int_{0}^{\varrho}(1-\theta)^{\frac{\alpha}{\beta}-\frac{\alpha(d-\lambda)}{\beta q_{3}}-1}
\theta^{-\frac{\chi_{2}+\chi_{3}}{2}}\left((t\theta)^{\frac{\chi_{3}}{2}}
\big\|\nabla\tilde{w}(t\theta)\big\|_{\mathcal{M}_{q_{3},\lambda}(\mathbb{R}^{d})}
+(t\theta)^{\frac{\chi_{2}}{2}}\big\|\tilde{v}(t\theta)\big\|_{\mathcal{M}_{q_{2},\lambda}(\mathbb{R}^{d})}
\right)d\theta,
\end{align*}
and
\begin{align*}
H_{2}(t)&=\kappa\int_{0}^{\varrho}(1-\theta)^{\alpha-\frac{\alpha}{\beta}-\frac{\alpha(d-\lambda)}{\beta q_{1}}-1}\theta^{-\frac{\chi_{1}+\chi_{3}}{2}}\left((t\theta)^{\frac{\chi_{3}}{2}}\big\|\nabla\tilde{w}(t\theta)\big\|
_{\mathcal{M}_{q_{3},\lambda}(\mathbb{R}^{d})}+(t\theta)^{\frac{\chi_{1}}{2}}\big\|\tilde{u}(t\theta)\big\|
_{\mathcal{M}_{q_{1},\lambda}(\mathbb{R}^{d})}
\right)d\theta\\
&\quad +\kappa\int_{0}^{\varrho}(1-\theta)^{2\alpha-\frac{3\alpha}{\beta}-\frac{\alpha(d-\lambda)}{\beta q_{2}}+\frac{\alpha(d-\lambda)}{\beta q_{3}}-1}\theta^{-\frac{\chi_{2}}{2}}\left((t\theta)^{\frac{\chi_{2}}{2}}\big\|\tilde{v}(t\theta)\big\|_{q_{2},\lambda}+\big\|\tilde{w}(t\theta)
\big\|_{\mathcal{M}^{2-\beta}_{\frac{d-\lambda}{2-\beta},\lambda}}\right)d\theta\\
&\quad +\kappa\int_{0}^{\varrho}(1-\theta)^{\alpha-\frac{\alpha(d-\lambda)}{\beta q_{1}}-\frac{\alpha(d-\lambda)}{\beta q_{3}}-1}\theta^{-\frac{\chi_{1}+\chi_{3}}{2}}\left((t\theta)^{\frac{\chi_{3}}{2}}\big\|\nabla\tilde{w}(t\theta)\big\|
_{\mathcal{M}_{q_{3},\lambda}(\mathbb{R}^{d})}+
(t\theta)^{\frac{\chi_{1}}{2}}\big\|\tilde{u}(t\theta)\big\|_{\mathcal{M}_{q_{1},\lambda}(\mathbb{R}^{d})}\right)d\theta\\
&\quad +\kappa\int_{0}^{\varrho}(1-\theta)^{2\alpha-\frac{2\alpha}{\beta}-\frac{\alpha(d-\lambda)}{\beta q_{2}}-1}\theta^{-\frac{\chi_{2}}{2}}\left(\big\|\tilde{w}(t\theta)
\big\|_{\mathcal{M}^{2-\beta}_{\frac{d-\lambda}{2-\beta},\lambda}}
+(t\theta)^{\frac{\chi_{2}}{2}}\big\|\tilde{v}(t\theta)\big\|_{\mathcal{M}_{q_{2},\lambda}(\mathbb{R}^{d})}
\right)d\theta.\\
\end{align*}
{\bf Step 3.}
Therefore, if $\lim_{t\rightarrow\infty}f(t)=0$, from above estimates, we have
$$
\limsup_{t\rightarrow\infty}g(t)\leq C\kappa\left(1+H(\varrho)\right)\limsup_{t\rightarrow\infty}g(t),
$$
where
\begin{align*}
H(\varrho)&=\int_{0}^{\varrho}(1-\theta)^{2\alpha-\frac{2\alpha}{\beta}-\frac{2\alpha(d-\lambda)}{\beta q_{1}}-1}\theta^{-\chi_{1}}d\theta +\int_{0}^{\varrho}(1-\theta)^{2\alpha-\frac{2\alpha}{\beta}-\frac{2\alpha(d-\lambda)}{\beta q_{2}}-1}\theta^{-\frac{\chi_{2}}{2}}d\theta\\
&\quad +\int_{0}^{\varrho}(1-\theta)^{3\alpha-\frac{3\alpha}{\beta}-\frac{\alpha(d-\lambda)}{\beta q_{1}}-\frac{\alpha(d-\lambda)}{\beta q_{2}}-1}\theta^{-\frac{\chi_{1}+\chi_{2}}{2}}d\theta\\
&\quad +\int_{0}^{\varrho}(1-\theta)^{2\alpha-\frac{\alpha}{\beta}-\frac{\alpha(d-\lambda)}{\beta q_{2}}-\frac{\alpha(d-\lambda)}{\beta q_{3}}-1}\theta^{-\frac{\chi_{2}+\chi_{3}}{2}}d\theta\\
&\quad+\int_{0}^{\varrho}(1-\theta)^{\alpha-\frac{\alpha(d-\lambda)}{\beta q_{1}}-\frac{\alpha(d-\lambda)}{\beta q_{3}}-1}
\theta^{-\frac{\chi_{1}+\chi_{3}}{2}}d\theta+\int_{0}^{\varrho}(1-\theta)^{2\alpha-\frac{2\alpha}{\beta}-\frac{2\alpha(d-\lambda)}{\beta q_{2}}-1}
\theta^{-\frac{\chi_{2}}{2}}d\theta\\
&\quad +\int_{0}^{\varrho}(1-\theta)^{\alpha-\frac{\alpha}{\beta}-\frac{\alpha(d-\lambda)}{\beta}
(\frac{2}{q_{1}}-\frac{1}{q_{1}})-1}\theta^{-\chi_{1}}d\theta\\
&\quad +\int_{0}^{\varrho}(1-\theta)^{\alpha-\frac{\alpha(d-\lambda)}{\beta}(\frac{1}{q_{2}}+\frac{1}{d-\lambda}-\frac{1}{q_{1}})-1}
\theta^{-\frac{\chi_{2}}{2}}d\theta\\
&\quad +\int_{0}^{\varrho}(1-\theta)^{\alpha-\frac{\alpha}{\beta}-\frac{\alpha(d-\lambda)}{\beta q_{1}}-1}\theta^{-\frac{\chi_{1}+\chi_{2}}{2}}d\theta+\int_{0}^{\varrho}(1-\theta)^{\frac{\alpha}{\beta}-\frac{\alpha(d-\lambda)}{\beta q_{3}}-1}
\theta^{-\frac{\chi_{2}+\chi_{3}}{2}}d\theta\\
&\quad +\int_{0}^{\varrho}(1-\theta)^{\alpha-\frac{\alpha}{\beta}-\frac{\alpha(d-\lambda)}{\beta q_{1}}-1}\theta^{-\frac{\chi_{1}+\chi_{3}}{2}}d\theta\\
&\quad +\int_{0}^{\varrho}(1-\theta)^{2\alpha-\frac{3\alpha}{\beta}-\frac{\alpha(d-\lambda)}{\beta q_{2}}+\frac{\alpha(d-\lambda)}{\beta q_{3}}-1}\theta^{-\frac{\chi_{2}}{2}}d\theta\\
&\quad +\int_{0}^{\varrho}(1-\theta)^{\alpha-\frac{\alpha(d-\lambda)}{\beta q_{1}}-\frac{\alpha(d-\lambda)}{\beta q_{3}}-1}\theta^{-\frac{\chi_{1}+\chi_{3}}{2}}d\theta+\int_{0}^{\varrho}(1-\theta)^{2\alpha-\frac{2\alpha}{\beta}-\frac{\alpha(d-\lambda)}{\beta q_{2}}-1}\theta^{-\frac{\chi_{2}}{2}}d\theta.
\end{align*}
By the proposition of the Beta function, we have $\lim_{\varrho\rightarrow 0^{+}}H(\varrho)=0$ and the constant $\kappa$
is small enough, thus, we obtain that $\limsup_{t\rightarrow\infty}g(t)=0$.

{\bf Step 4.}
On the other hand, if $\limsup_{t\rightarrow\infty}g(t)=0$, note that
\begin{align*}
&\left(\mathcal{S}_{\alpha,\beta}(t)\left(u_{10}-u_{20},v_{10}-v_{20},w_{10}-w_{20}\right)\right)\\
&=
\left(u_{1}-u_{2},v_{1}-v_{2},w_{1}-w_{2}\right)-\left(\tilde{u},\tilde{v},\tilde{w}\right),
\end{align*}
thus, we get
\begin{align*}
\limsup_{t\rightarrow\infty}f(t)\leq C\kappa\left(1+H(\varrho)\right)\limsup_{t\rightarrow\infty}g(t)=0
\end{align*}
The proof is thus completed.
\end{proof}
\section{Conclusions}
In this article, we establish the global existence and analyze the behavior of a class of time-space fractional Navier-Stokes-Keller-Segel systems \eqref{Eq:K-S-N-S} in Besov-Morrey spaces by using the embedding theorem of Besov-Morrey spaces, multiplier theory, real interpolation techniques, and the Hardy-Littlewood inequality in Morrey spaces. This system is a generalized form of the Navier-Stokes-Keller-Segel system established in \cite{M.H. Yang}. Due to the lack of smoothness at $\xi=0$ of the symbol of the fractional heat kernel $e^{-t(-\Delta)^{\frac{\beta}{2}}}$, we establish a time-space estimate in Besov-Morrey spaces for the fractional heat kernel $e^{-t(-\Delta)^{\frac{\beta}{2}}}$ using multiplier theory in Besov-Morrey spaces. Specifically, when $\beta=2$, it corresponds to the time-space estimate of the classical heat kernel $e^{t\Delta}$.

Due to the lack of semigroup structure of the solution operator $\mathcal{S}_{\alpha,\beta}(t)$ and $\mathcal{P}_{\alpha,\beta}(t)$, we cannot use the smoothing effect of the heat kernel in Besov-Morrey spaces as in \eqref{HKSE}, which is fundamentally different from the methods in \cite{M.H. Yang, J. Zhang}. It is worth noting that our approach is also applicable to the systems in \cite{M.H. Yang, J. Zhang}. Moreover, since $C_{0}^{\infty}(\mathbb{R}^{d})$ is not dense in Besov-Morrey spaces, and the lack of smoothness at $\xi=0$ of $e^{-t|\xi|^{\beta}}$ leads to the inability of $e^{-t(-\Delta)^{\frac{\beta}{2}}}$ to act as a multiplier on $\mathcal{S}'$, the solution of \eqref{Eq:K-S-N-S} at $t=0$ weak* convergence is not clear, which is worthy of further discussion.

\noindent{\bf Declaration of competing interest}\\
The authors declare that they have no competing interests.\\
\noindent{\bf Data availability}\\
No data was used for the research described in the article.\\
\noindent{\bf Acknowledgements}\\
This work was supported by National Natural Science Foundation of China (12471172) and Fundo para o Desenvolvimento das Ci\^{e}ncias e da Tecnologia of Macau (No. 0092/2022/A).

\end{document}